\documentclass{amsart}
\usepackage[margin=2cm]{geometry}
\usepackage{amsmath,amssymb,amsfonts,mathtools,tikz,tabularx,arydshln,diagbox,
  mathrsfs,algorithm,algorithmic,graphicx,etoolbox,soul,csquotes,
  pdflscape, cprotect, textcomp, pgfplots, hhline, multirow,hyperref}
\usetikzlibrary{calc}
\usetikzlibrary{shapes.geometric}
\usepackage[foot]{amsaddr}
\usepackage[font=itshape]{quoting}
\usepackage[utf8]{inputenc} 
\usepackage[T1]{fontenc}
\usepackage[toc,page]{appendix} 
\usepackage[space,noadjust,nocompress]{cite}
\usepackage[mathlines]{lineno}
\usepackage[hypcap=false]{subcaption}
\usepackage[hypcap=false]{caption}
\usepackage{newfloat}
\DeclareFloatingEnvironment[]{FloatingAlgorithm}

\newcommand{\mtx}[1]{\boldsymbol{#1}}
\newcommand{\vct}[1]{\boldsymbol{#1}}

\newtheorem{remark}{Remark}

\newtheorem{definition}{Definition}
\newtheorem{lemma}{Lemma}

\title[An iterative 3D HPS-discretized Helmholtz solver]{An iterative solver for the HPS discretization applied to three dimensional Helmholtz problems}

\author{Jos\'e Pablo Lucero Lorca}
\thanks{\url{https://pablo.world/math}}
\email{pablo.lucero@colorado.edu}
\author{Natalie Beams}
\author{Damien Beecroft}
\author{Adrianna Gillman}
\thanks{\url{https://www.colorado.edu/amath/adrianna-gillman}}
\email{adrianna.gillman@colorado.edu}
\address[J. P. Lucero Lorca, A. Gillman]{Department of Applied Mathematics, University of Colorado at Boulder}
\thanks{Submitted to the editors 3 December 2021. This work was funded by Total Energies and NSF.}

\makeatletter
\renewcommand{\email}[2][]{%
  \ifx\emails\@empty\relax\else{\g@addto@macro\emails{,\space}}\fi%
  \@ifnotempty{#1}{\g@addto@macro\emails{\textrm{(#1)}\space}}%
  \g@addto@macro\emails{#2}%
}
\makeatother

\newcommand\scalemath[2]{\scalebox{#1}{\mbox{\ensuremath{\displaystyle#2}}}}

\newcommand\w[1]{\makebox[2.5em]{$#1$}}

\newcommand{\mtwo}[4]{\left[\begin{array}{cc} #1 & #2 \\ #3 & #4 \end{array}\right]}
\newcommand{\vtwo}[2]{\left[\begin{array}{cc} #1 \\ #2 \end{array}\right]}

\allowdisplaybreaks

\begin{document}

\begin{abstract}
  This manuscript presents an efficient solver for the linear system
  that arises from the Hierarchical Poincar\'e-Steklov (HPS)
  discretization of three dimensional variable coefficient Helmholtz
  problems.  Previous work on the HPS method has tied it with a direct
  solver.  This work is the first efficient iterative solver for the
  linear system that results from the HPS discretization.  The solution
  technique utilizes \texttt{GMRES} coupled with a locally homogenized
  block-Jacobi preconditioner. The local nature of the discretization
  and preconditioner naturally yield the matrix-free application of the
  linear system.  Numerical results illustrate the performance of the
  solution technique.  This includes an experiment where a problem
  approximately $100$ wavelengths in each direction that requires more
  than a billion unknowns to achieve approximately 4 digits of accuracy
  takes less than $20$ minutes to solve.
  
  \smallskip
  
  \textbf{Key words.} Helmholtz, HPS, block-Jacobi, Domain
  Decomposition, Poincar\'e-Steklov, GMRES.

  \textbf{AMS subject classifications.} 65N22, 65N35, 65N55, 65F05

\end{abstract}

\maketitle

\section{Introduction}

The efficient simulation of wave propagation phenomena in 3D
heterogeneous media is of great research interest in many areas.  This
document considers wave propagation problems that are captured by the
variable coefficient Helmholtz problem with impedance boundary
conditions given below
\begin{align}
  \begin{aligned}
    -\Delta u(\boldsymbol{x}) - \kappa^2 (1-b(\boldsymbol{x}))
    u(\boldsymbol{x}) =& s(\boldsymbol{x}), \quad \boldsymbol{x} \in
    \Omega \text{, and} \\
    \frac{\partial u}{\partial n} + i \eta u(\boldsymbol{x}) =&
    t(\boldsymbol{x}), \quad \boldsymbol{x} \in \partial\Omega.
  \end{aligned}\label{eqn:diffhelm}
\end{align}
where $\Omega = (0,1)^3 \subset \mathbb{R}^3$, $u(\boldsymbol{x}):
\Omega \rightarrow \mathbb{C}$ is the unknown solution, $\kappa \in
\mathbb{R}$ is the so-called \emph{wave number},
$b(\boldsymbol{x}):\Omega \rightarrow \mathbb{C}$ is a given smooth
scattering potential and $n$ is the outward normal unit vector to the
boundary of the domain. The functions $s(\boldsymbol{x}):\Omega
\rightarrow \mathbb{C}$ and $t(\boldsymbol{x}):\partial\Omega
\rightarrow \mathbb{C}$ are assumed to be smooth functions.  The usage
of impedance boundary conditions such as the one in
(\ref{eqn:diffhelm}) can be found in diffraction theory, acoustics,
seismic inversion, electromagnetism and others
\cite{KyurkchanSmirnova2016,
KinslerFreyCoppensSanders2000,Lighthill1978, ChandlerWilde1997}.  For
a detailed exposition on the state-of-the-art literature and
applications see \cite[\S1.1,\S1.2]{GrahamSauter2020} and references
therein.

This manuscript presents an efficient technique for solving the linear
system that results from the discretization of \eqref{eqn:diffhelm}
with the Hierarchical Poincar\'e-Steklov (HPS) method.  Roughly
speaking, the HPS method is a discretization technique based on local
``element''-wise spectral collocation.  Continuity of the solution and
the flux are enforced strongly at the interface between elements.  For
the case of Helmholtz equation, the continuity of incoming and
outgoing impedance data is enforced between elements.  As with other
element-wise discretization techniques, the matrix that results from
the HPS is sparse where all non-zero entries correspond to an element
interacting with itself or with a neighbor.  Such sparse matrices can
be applied to a vector in the so-called \emph{matrix-free} format which means
that each sub-matrix can be applied by sweeping over the elements and
never building the full matrix.

The efficient solution for the linear system resulting from the HPS
discretization presented in this paper utilizes a locally homogenized
block Jacobi preconditioned \texttt{GMRES} \cite{SaadSchultz1986}
solver. The application of the proposed block Jacobi preconditioner
and the matrix itself is done via matrix-free operations and exploits
the tensor product nature of the element wise discretization matrices.
The proposed preconditioner requires inversion of the Jacobi blocks
with a homogeneous coefficient at element level. The local nature of
the blocks and the preconditioner make the solution technique
naturally parallelizable in a distributed memory model. Numerical
results show that the solution technique is efficient and capable of
tackling mid-frequency problems with a billion degrees of freedom in
less than thirty minutes in parallel.

\subsection{Prior work on the HPS method} \label{sec:priorHPS}
The HPS method is a recently developed discretization technique
\cite{Martinsson2013, GillmanMartinsson2014,
GillmanBarnettMartinsson2015, HaoMartinsson2016,
GeldermansGillman2019, BeamsGillmanHewett2020}, based on local
spectral collocation where the continuity of the solution and the flux
is enforced via Poincar\'e-Steklov operators.  For general elliptic
problems, it is sufficient enough to use the Dirichlet-to-Neumann
operator \cite[Rem. 3.1]{Martinsson2009}. For Helmholtz problems, the
Impedance-to-Impedance (ItI) operator is used
\cite{GillmanBarnettMartinsson2015, BeamsGillmanHewett2020}.  By using
this unitary operator for the coupling of elements, the HPS method is
able to avoid artificial resonances and does not appear to observe the
so-called \emph{pollution effect} \cite{GillmanBarnettMartinsson2015}.
The paper \cite{BeckCanzaniMarzuola2021} provides some analysis
supporting the numerical results seen in practice.

\begin{remark}
  Roughly speaking, pollution
  \cite[Def. 2.3]{BabuskaIhlenburgPaikSauter1995} is the need to
  increase the number of degrees of freedom per wavelength as the wave
  number $\kappa$ grows in order to maintain a prescribed accuracy.
  This problem is known to happen when Galerkin finite element
  discretizations in two and more space dimensions are applied to
  Helmholtz problems \cite{BaylissGoldsteinTurkel1985,
    BabuskaIvoSauter1997, BabuskaIhlenburgPaikSauter1995}.
\end{remark}

Thus far the HPS method has only been applied to problems when coupled
with a nested dissection \cite{George1973} type direct solver.  For
two dimensional problems, this direct solver is easily parallelizable
using shared memory \cite{BeamsGillmanHewett2020} and integrated into
an adaptive version of the discretization technique
\cite{GeldermansGillman2019}.  Additionally, the computational cost of
the direct solver is $O(N^{3/2})$, where $N$ is the number of
discretization points, for two dimensional problems with a small
constant. For many problems the direct solver can be accelerated to
have a computational cost that scales linearly with respect $N$ in two
dimensions \cite{GillmanMartinsson2014} and $O(N^{4/3})$ in three
dimensions \cite{HaoMartinsson2016}.

While direct solvers have been proven to be robust for various
problems and parameters, they scale superlinearly in flops and storage
for three dimensional problems.  Additionally, there is a large amount
of communication in the construction of solver making it complicated
to get highly efficient parallel implementations.  The application of
the precomputed solver is less cumbersome in a parallel environment.
Essentially, a user has to decide if they have
enough right hand sides to justify the cost of building a direct
solver versus just using an iterative solver.  When solving very large
problems, in the order of the billion degrees of freedom for a single
right hand side, iterative methods are (in general) faster. 
For
instance, shifted-Laplace preconditioners
\cite{ErlanggaVuikOosterlee2004, KimLee1996} introduce an imaginary
shift $\epsilon$ to the frequency $\omega$ at subdomain level for the
preconditioner and show a complexity $\mathcal{O}(N^{4/3})$ for 3D
problems with $N=\mathcal{O}(\omega^3)$
\cite{CalandraGrattonPinelVasseur2012, CoolsRepsVanroose2014,
RiyantiKononovErlanggaVuikOosterleePlessix2007}. In the search for a
scalable Helmholtz solver, many techniques have been used including
efficient coarse problems to precondition Krylov solvers in the
context of multigrid methods and domain decomposition, local
eigenspaces for heterogeneous media, complex-symmetric least-square
formulations, numerical-asymptotic hybridization and block-Krylov
solvers (see \cite[\S4]{GanderZhang2019} and references therein for a
detailed description).  If a problem
is poorly conditioned, it may be beneficial to use a direct solver instead
of an iterative solver that may have trouble converging even with a 
preconditioner.

This manuscript presents the first iterative solution technique for
the linear system that arises from the HPS discretization.

\subsection{Related work} \label{sec:related}
There are several spectral collocation techniques that are similar to
the HPS discretization in spirit (such as
\cite{PfeifferKidderScheelTeukolsky2003}).  A thorough review of those
methods is provided in \cite{Martinsson2013}.

The block Jacobi preconditioner presented in this paper can be viewed
as a homogenized non-overlapping Schwarz domain decomposition
preconditioner.  Domain decomposition preconditioned solvers are
widely used for solving large PDE discretizations with a distributed
memory model parallelization.  These techniques rely on the algorithm
being able to perform the bulk of the calculation \emph{locally} in
each subdomain and limiting the amount of communication between
parallel processes (see \cite{SmithBjorstadGropp1996,
ToselliWidlund2005} and references therein). Whenever possible using
local tensor product formulations for the discretized operator and
preconditioner, matrix-free techniques can accelerate the
implementation and minimize memory footprint of domain decomposition
preconditioned solvers (e.g. \cite{Knyazev2001,
KronbichlerKormann2019}). By using a high order Chebyshev tensor
product local discretization and sparse inter-element operators, the
solution technique presented in this manuscript is domain decomposing
in nature and the implementation matrix-free to a large extent.

The implementation of the algorithms presented in this paper are
parallelized using the \texttt{PETSC} library \cite{petsc-web-page,
petsc-user-ref, petsc-efficient} with a distributed memory model.

\subsection{Outline}
The manuscript begins by describing the HPS discretization (in section
\ref{sec:dis}) and how the linear system that results from that
discretization can be applied efficiently in a matrix-free manner in
section \ref{sec:sparse}.  Then section \ref{sec:pre} presents the
proposed solution technique, including a block Jacobi preconditioner
and its fast application.  Next section \ref{sec:par} presents the
technique for implementing the solver on a distributed memory
platform.  Numerical results illustrating the performance of the
solver are presented in section \ref{sec:numerics}.  Finally, the
paper concludes with a summary highlighting the key features and the
impact of the numerical results in section \ref{sec:con}.

\section{Discretizing via the HPS method}
\label{sec:dis}
The HPS method applied to \eqref{eqn:diffhelm} is based on
partitioning the geometry $\Omega$ into a collection of small boxes.
In previous papers on the HPS method, these small boxes have been
called \emph{leaf} boxes but the term \emph{element} from the finite
element literature can also be applied to them.  Each element is
discretized with a modified spectral collocation method and continuity
of impedance data is used to \emph{glue} the boxes together.  This
section reviews the discretization technique.

\begin{remark}
For simplicity of presentation, all the leaf boxes are taken to be the
same size in this manuscript.  In other words, we consider a uniform
mesh or domain partitioning for simplicity of presentation.  The
technique can easily be extended to a nonuniform mesh with the use of
interpolation operators and maintaining the appropriate ratio in size
between neighboring boxes (as presented in
\cite{GeldermansGillman2019}).
\end{remark}

\subsection{Leaf discretization}\label{sec:leafdiscretization}
This section describes the discretization technique that is applied to
each of the leaf boxes. This corresponds to discretizing equation
\eqref{eqn:diffhelm} with a fictitious impedance boundary condition on
each leaf box via a modified classical spectral collocation
technique. The modified spectral collocation technique was first
presented in \cite{GeldermansGillman2019} for two dimensional
problems.

Consider the leaf box (or element) $\Omega^\tau \subset \Omega$ such
as the one illustrated in Figure \ref{fig:chebpoints}.  The modified
spectral collocation technique begins with the classic $n_c\times
n_c\times n_c$ tensor product Chebyshev grid and the corresponding
standard differentiation matrices $\widetilde{\mtx{D}}_x$,
$\widetilde{\mtx{D}}_y$ and $\widetilde{\mtx{D}}_z$, as defined in
\cite{Trefethen2000,Boyd2001}, on $\Omega^\tau$. Here $n_c$ denotes
the number of Chebyshev nodes in one direction. The entries of
$\widetilde{\mtx{D}}_x$, $\widetilde{\mtx{D}}_y$ and
$\widetilde{\mtx{D}}_z$ corresponding to the interaction of the corner
and edge points with the points on the interior of $\Omega^\tau$ are
zero thanks to the tensor product basis.  Thus these discretization
nodes can be removed without impacting the accuracy of the
corresponding derivative operators.  Let $\mtx{D}_x$, $\mtx{D}_y$ and
$\mtx{D}_z$ denote the submatrices of $\widetilde{\mtx{D}}_x$,
$\widetilde{\mtx{D}}_y$ and $\widetilde{\mtx{D}}_z$ that approximate
the derivatives of functions with a basis that does not have
interpolation nodes at the corner and edge points entries.

Based on a tensor product grid without corner and edge points, let
$I_i^\tau$ denote the index vector corresponding to points in the
interior of $\Omega^\tau$ and $I_b^\tau$ denote the index vector
corresponding to points on the faces at the boundary of
$\Omega^\tau$. Figure \ref{fig:chebpoints} illustrates this local
ordering on a leaf box with $n_c = 8$.  Let $n$, $n_b$ and $n_i$
denote the number of discretization points on a leaf box, the number
of discretization points on the boundary of a leaf box and the number
of discretization points in the interior of a leaf box, respectively.
Thus the total number of discretization points on a leaf box is $n =
n_i+n_b$.  For the three-dimensional discretization, $n_i = (n_c-2)^3$
and $n_b = 6(n_c-2)^2$. Ordering the points in the interior first, the
local indexing of the $n$ discretization points associated with
$\Omega^\tau$ is $I^\tau = \left[I^\tau_i I^\tau_b\right]$.

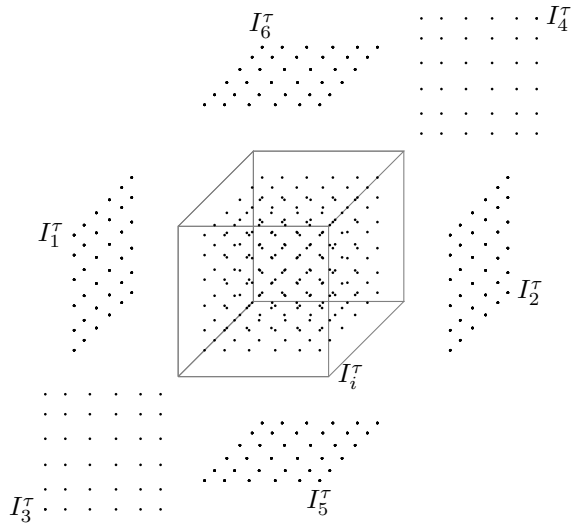
\begin{figure}
  \centering
  \def\deg{8}
  \def\ps{0.25pt}
  \def\degm{7}
  \begin{tikzpicture}[scale=2.]
   \draw[color=gray] (0,0) -- (1,0) -- (1,1) -- (0,1) -- (0,0);
   \draw[color=gray] (0.5,0.5) -- (1.5,0.5) -- (1.5,1.5) -- (0.5,1.5) -- (0.5,0.5);
   \draw[color=gray] (0,0) -- (1,0) -- (1.5,0.5) -- (0.5,0.5) -- (0,0);
   \draw[color=gray] (0,1) -- (1,1) -- (1.5,1.5) -- (0.5,1.5) -- (0,1);
   \draw[color=gray] (0,0) -- (0.5,0.5) -- (0.5,1.5) -- (0,1) -- (0,0);
   \draw (1,0) node[anchor=west] {$I_i^\tau$};

   \foreach \x in {2,...,\degm}
   \foreach \y in {2,...,\degm}
   {
     \fill[black]({1.+1-cos(180.*\x/(\deg+1))/2.},{1.+1-cos(180.*\y/(\deg+1))/2.})circle(\ps);
   }
   \fill[black]({1.+1-cos(180.*(\deg-1)/(\deg+1))/2.},{1.+1-cos(180.*(\deg-1)/(\deg+1))/2.})circle(\ps) node[anchor=west] {$I_4^\tau$} ;

   \foreach \z in {2,...,\degm}
   \foreach \x in {2,...,\degm}
   \foreach \y in {2,...,\degm}
   {
     \fill[black]({0.25-cos(180.*\x/(\deg+1))/2.+0.5-cos(180.*\z/(\deg+1))/4.},{0.25-cos(180.*\y/(\deg+1))/2.+0.5-cos(180.*\z/(\deg+1))/4.})circle(\ps);
     \fill[black]({0.25-cos(180.*\x/(\deg+1))/2.+0.5-cos(180.*\z/(\deg+1))/4.},{-0.75+0.25-cos(180.*\z/(\deg+1))/4.})circle(\ps);
     \fill[black]({0.75+0.25-cos(180.*\z/(\deg+1))/4.+1},{0.25-cos(180.*\y/(\deg+1))/2.+0.5-cos(180.*\z/(\deg+1))/4.})circle(\ps);
     \fill[black]({0.25-cos(180.*\x/(\deg+1))/2.+0.5-cos(180.*\z/(\deg+1))/4.},{0.75+0.25-cos(180.*\z/(\deg+1))/4.+1})circle(\ps);
     \fill[black]({-0.75+0.25-cos(180.*\z/(\deg+1))/4.},{0.25-cos(180.*\y/(\deg+1))/2.+0.5-cos(180.*\z/(\deg+1))/4.})circle(\ps);
   }
   \fill[black]({0.25-cos(180.*(\deg-1)/(\deg+1))/2.+0.5-cos(180.*(2.)/(\deg+1))/4.},{-0.75+0.25-cos(180.*(2.)/(\deg+1))/4.})circle(\ps) node[anchor=north] {$I_5^\tau$};
   \fill[black]({0.75+0.25-cos(180.*(\deg-1)/(\deg+1))/4.+1},{0.25-cos(180.*2./(\deg+1))/2.+0.5-cos(180.*(\deg-1)/(\deg+1))/4.})circle(\ps) node[anchor=west] {$I_2^\tau$};
   \fill[black]({0.25-cos(180.*2./(\deg+1))/2.+0.5-cos(180.*(\deg-1)/(\deg+1))/4.},{0.75+0.25-cos(180.*(\deg-1)/(\deg+1))/4.+1})circle(\ps) node[anchor=south] {$I_6^\tau$} ;
   \fill[black]({-0.75+0.25-cos(180.*2./(\deg+1))/4.},{0.25-cos(180.*(\deg-1)/(\deg+1))/2.+0.5-cos(180.*2./(\deg+1))/4.})circle(\ps) node[anchor=east] {$I_1^\tau$} ;
   
   \foreach \x in {2,...,\degm}
   \foreach \y in {2,...,\degm}
   {
     \fill[black]({-1.+0.5-cos(180.*\x/(\deg+1))/2.},{-1.+0.5-cos(180.*\y/(\deg+1))/2.})circle(\ps);
   }
   \fill[black]({-1.+0.5-cos(180.*2./(\deg+1))/2.},{-1.+0.5-cos(180.*2./(\deg+1))/2.})circle(\ps) node[anchor=east] {$I_3^\tau$};
   
  \end{tikzpicture}
  \caption{Illustration of the degree 8 Chebyshev tensor-product
    points on a cube $\Omega^\tau$ and the local numbering of the nodes.
    $I^\tau_i$ denotes the interior nodes and $I^\tau_j$ for $j =1,\ldots,
    6$ denotes the nodes on face $j$.}
  \label{fig:chebpoints}
\end{figure} 

We approximate the solution to \eqref{eqn:diffhelm} on $\Omega^\tau$
using the \textit{new} spectral collocation operators by
\begin{equation}
  \boldsymbol{A}_c^\tau := -\boldsymbol{D}_x^2 -
  \boldsymbol{D}_y^2 - \boldsymbol{D}_z^2 -
  \boldsymbol{C}^\tau
\end{equation}
where $\boldsymbol{C}^\tau$ is the diagonal matrix with entries
$\{\kappa^2 (1 - b(\boldsymbol{x}_k))\}_{k=1}^{n}$.

To construct the operator for enforcing the impedance boundary
condition, we first order the indices in $I^\tau_b$ according to the
faces.  Specifically, $I_b^\tau = [I_1^\tau, I_2^\tau, I_3^\tau,
I_4^\tau, I_5^\tau, I_6^\tau]$ where $I_1$ denotes the index of the
points on the left of the box, $I_2$ on the right, etc.  Figure
\ref{fig:chebpoints} illustrates a numbering of the faces for a box
$\Omega^\tau$.  The operator that approximates the normal derivative
on the boundary of the box $\Omega^\tau$ is the $n_b \times n$ matrix
$\mtx{N}$ given by
\begin{equation}
  \boldsymbol{N} = \left(
    \begin{matrix}
      -\boldsymbol{D}_x(I^\tau_1,I^\tau) \\
       \boldsymbol{D}_x(I^\tau_2,I^\tau) \\
      -\boldsymbol{D}_y(I^\tau_3,I^\tau) \\
       \boldsymbol{D}_y(I^\tau_4,I^\tau) \\
      -\boldsymbol{D}_z(I^\tau_5,I^\tau) \\
       \boldsymbol{D}_z(I^\tau_6,I^\tau) \\      
    \end{matrix}
  \right). \label{eqn:flux}
\end{equation}
Thus the outgoing impedance operator is approximated by the $n_b\times
n$ matrix $\mtx{F}^\tau$ defined by
\begin{equation*}
  \mtx{F}^\tau = \boldsymbol{N} + i \eta
  \boldsymbol{I}_{n}\left(I^\tau_b,I^\tau\right),
\end{equation*}
where $\boldsymbol{I}_{n}$ is the identity matrix of size $n \times
n$.

With this, the linear system that approximates the solution to
(\ref{eqn:diffhelm}) fictitious boundary data $\hat{f}(x)$ on a box
$\Omega^\tau$ is given by
\begin{equation}
\mtx{A}^\tau \vtwo{\vct{u}_i}{\vct{u}_b}= \left[\begin{array}{l}
\mtx{A}_c^\tau(I_i^\tau,I^\tau) \\ 
\mtx{F}^\tau
\end{array}\right]
\vtwo{\vct{u}_i}{\vct{u}_b}= \vtwo{\vct{s}}{\vct{\hat{f}}}
\label{eqn:leaf}
\end{equation}
where the vector $\vct{u}$ denotes the approximate solution,
$\vct{u}_i = \vct{u}(I_i^\tau)$, $\vct{u}_b = \vct{u}(I_b^\tau)$,
$\vct{\hat{f}}$ is the fictitious impedance boundary data at the
discretization points on $\partial \Omega^\tau$, and $\vct{s}$ is the
right hand side of the differential operator in \eqref{eqn:diffhelm}
evaluated at the discretization points in the interior of
$\Omega^\tau$.
 
\subsection{Communication between leaf boxes}
\label{sec:comm}
While the previous section provides a local discretization, it does
not provide a way for the leaves (elements) to communicate information
between each other.  As mentioned previously, the HPS method
communicates information between elements by enforcing conditions on
the continuity of the approximate solution and its flux through shared
boundaries.  For Helmholtz problems, this is done via continuity of
impedance boundary data.  In other words, the incoming impedance data
for one element is equal to (less a sign) the outgoing impedance data
from its neighbor along the shared face.  Incoming and outgoing
impedance boundary data and the relationship are defined as follows:
\begin{definition}
    Fix $\eta \in \mathbb{C}$, and $\mathbb{R}(\eta) \ne 0$. Let
  \begin{align}
    \begin{aligned}
      f :=& \frac{\partial u}{\partial n} + i \eta u \big|_\Gamma \text{, and} \\
      g :=& \frac{\partial u}{\partial n} - i \eta u \big|_\Gamma
    \end{aligned}
  \end{align}
  be Robin traces of $u$ at an arbitrary interface
  $\Gamma$. $\frac{\partial u}{\partial n}$ denotes the normal
  derivative of $u$ in the direction of the normal vector $n$. We refer
  to $f$ and $g$ as the \emph{incoming} and \emph{outgoing}
  (respectively) impedance data.
\end{definition}

To illustrated how this condition is enforced in the HPS method
consider the two neighboring boxes $\alpha $ and $\beta$ in Figure
\ref{fig:2leaves}.  The shared interface is $\Gamma$.  Note that
$\frac{\partial u}{\partial n_\alpha} = - \frac{\partial u}{\partial
n_\beta} $ and $ u_\alpha \big|_\Gamma = u_\beta \big|_\Gamma$.  With
this information, it is easy to show that
\begin{equation}
  f^\alpha_\Gamma = -g^\beta_\Gamma.
  \label{eq:inter}
\end{equation}
      
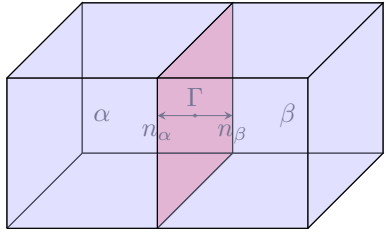
\begin{figure}
  \centering
  \begin{tikzpicture}[scale=2,draw=black]
    \draw (0.25,0.25+0.5) node {};
    \draw (0.5+0.25,0.25+0.5) node[anchor=east,color=black] {$\alpha$};
    \filldraw[fill=red!40] (1,0) -- (1.5,0.5) -- (1.5,1.5) -- (1,1) -- (1,0);
    \draw (1+0.25,0.25+0.5) circle(0.25pt) node[anchor=south,color=black] {$\Gamma$} ;
    \draw[-stealth] (1+0.25,0.25+0.5) -- (1+0.25+0.25,0.25+0.5) node[anchor=north,color=black] {$n_\beta$};
    \draw[-stealth] (1+0.25,0.25+0.5) -- (1+0.25-0.25,0.25+0.5) node[anchor=north,color=black] {$n_\alpha$};
    \filldraw[fill=blue!20,fill opacity=0.35] (0.5,0.5) -- (1.5,0.5) -- (1.5,1.5) -- (0.5,1.5) -- (0.5,0.5);
    \filldraw[fill=blue!20,fill opacity=0.35] (0,0) -- (1,0) -- (1.5,0.5) -- (0.5,0.5) -- (0,0);
    \filldraw[fill=blue!20,fill opacity=0.35] (0,0) -- (0.5,0.5) -- (0.5,1.5) -- (0,1) -- (0,0);
    \filldraw[fill=blue!20,fill opacity=0.35] (0,0) -- (1,0) -- (1,1) -- (0,1) -- (0,0);
    \filldraw[fill=blue!20,fill opacity=0.35] (0,1) -- (1,1) -- (1.5,1.5) -- (0.5,1.5) -- (0,1);
    \draw (1.5+0.25,0.25+0.5) node[anchor=west,color=black] {$\beta$};
    \filldraw[fill=blue!20,fill opacity=0.35] (1.5,0.5) -- (2.5,0.5) -- (2.5,1.5) -- (1.5,1.5) -- (1.5,0.5) -- cycle;
    \filldraw[fill=blue!20,fill opacity=0.35] (1,0) -- (2,0) -- (2.5,0.5) -- (1.5,0.5) -- (1,0) -- cycle;
    \filldraw[fill=blue!20,fill opacity=0.35] (1,1) -- (2,1) -- (2.5,1.5) -- (1.5,1.5) -- (1,1) -- cycle;
    \filldraw[fill=blue!20,fill opacity=0.35] (2,0) -- (2.5,0.5) -- (2.5,1.5) -- (2,1) -- (2,0);
    \filldraw[fill=blue!20,fill opacity=0.35] (1,0) -- (2,0) -- (2,1) -- (1,1) -- (1,0) -- cycle;
  \end{tikzpicture}
  \caption{Illustration of two leaf boxes sharing an interface.}
  \label{fig:2leaves}
\end{figure}

The HPS discretization enforces (\ref{eq:inter}) strongly.  The
spectral collocation operator approximating the outgoing impedance data on any leaf can be
constructed via the same techniques presented in section
\ref{sec:leafdiscretization}.  Specifically, let $\mtx{N}$ denote the
Chebyshev differentiation matrix that approximates the flux as
presented in equation \eqref{eqn:flux}.  Then the outgoing impedance
operator is
\begin{equation}
  \mtx{G}^\tau = \boldsymbol{N} - i \eta
  \boldsymbol{I}_{n}\left(I^\tau_b,I^\tau\right),\label{eqn:G}
\end{equation}
where $\boldsymbol{I}_{n}$ is the $n \times n$ identity matrix, and
$\eta$ is an impedance parameter.  In practice, we set $\eta =
\kappa$.

So in the full HPS discretization, the ficticuous boundary condition
gets replaced with either the true boundary condition or the 
equation enforcing the continuity of the impedance boundary
data from (\ref{eq:inter}).

For example, consider the task of applying the HPS method
to solve (\ref{eqn:diffhelm}) on the geometry illustrated in Figure \ref{fig:2leaves}.  The shared interface $\Gamma$ is the second face on $\alpha$
and the first face on $\beta$ according
to the numbering in Figure \ref{fig:chebpoints}.  Using an extra set of 
discretization points on these faces gives us that the equations 
that enforce the interface condition are
\begin{align*}
  \begin{cases}
  \mtx{F}^\alpha(I_2^\alpha,I^\alpha) \vct{u}^\alpha + \mtx{G}^\beta(I_1^\beta,I^\beta) \vct{u}^\beta =& \vct{\emptyset}(I_2^\alpha,1), \text{ and} \\
  \mtx{G}^\alpha(I_2^\alpha,I^\alpha) \vct{u}^\alpha + \mtx{F}^\beta(I_1^\beta,I^\beta) \vct{u}^\beta =& \vct{\emptyset}(I_1^\beta,1).
  \end{cases}
\end{align*}

Thus the linear
system that results from the discretization is 
\begin{align}
  \begin{aligned}
    \left(
      \begin{matrix}
        \left[\begin{array}{l}
\mtx{A}_c^\alpha(I_i^\alpha,I^\alpha) \\ 
                \mtx{F}^\alpha(I_1^\alpha,I^\alpha) \\
                \mtx{F}^\alpha(I_2^\alpha,I^\alpha) \\
                \mtx{F}^\alpha(I_3^\alpha,I^\alpha) \\
                \mtx{F}^\alpha(I_4^\alpha,I^\alpha) \\
                \mtx{F}^\alpha(I_5^\alpha,I^\alpha) \\
                \mtx{F}^\alpha(I_6^\alpha,I^\alpha) 
\end{array}\right]
&
\left[\begin{array}{l}
        \mtx{\emptyset}(I_i^\alpha,I^\beta) \\ 
        \mtx{\emptyset}(I_1^\alpha,I^\beta) \\
        \mtx{G}^\beta(I_1^\beta,I^\beta) \\
        \mtx{\emptyset}(I_3^\alpha,I^\beta) \\
        \mtx{\emptyset}(I_4^\alpha,I^\beta) \\
        \mtx{\emptyset}(I_5^\alpha,I^\beta) \\
        \mtx{\emptyset}(I_6^\alpha,I^\beta)
      \end{array}\right]
\\
\left[\begin{array}{l}
        \mtx{\emptyset}(I_i^\beta,I^\alpha) \\ 
        \mtx{G}^\alpha(I_2^\alpha,I^\alpha) \\
        \mtx{\emptyset}(I_2^\beta,I^\alpha)\\
        \mtx{\emptyset}(I_3^\beta,I^\alpha) \\
        \mtx{\emptyset}(I_4^\beta,I^\alpha) \\
        \mtx{\emptyset}(I_5^\beta,I^\alpha) \\
        \mtx{\emptyset}(I_6^\beta,I^\alpha)
      \end{array}\right]
&
 \left[\begin{array}{l}
\mtx{A}_c^\beta(I_i^\beta,I^\beta) \\ 
                \mtx{F}^\beta(I_1^\beta,I^\beta) \\
                \mtx{F}^\beta(I_2^\beta,I^\beta) \\
                \mtx{F}^\beta(I_3^\beta,I^\beta) \\
                \mtx{F}^\beta(I_4^\beta,I^\beta) \\
                \mtx{F}^\beta(I_5^\beta,I^\beta) \\
                \mtx{F}^\beta(I_6^\beta,I^\beta) 
\end{array}\right]
      \end{matrix}\right)
    \left(
      \begin{matrix}
        \left[\begin{array}{l}
                \vct{u}_i^\alpha \\ 
                \vct{u}_1^\alpha \\
                \vct{u}_2^\alpha \\
                \vct{u}_3^\alpha \\
                \vct{u}_4^\alpha \\
                \vct{u}_5^\alpha \\
                \vct{u}_6^\alpha 
\end{array}\right]
                \\
 \left[\begin{array}{l}
                \vct{u}_i^\beta \\ 
                \vct{u}_1^\beta \\
                \vct{u}_2^\beta \\
                \vct{u}_3^\beta \\
                \vct{u}_4^\beta \\
                \vct{u}_5^\beta \\
                \vct{u}_6^\beta 
\end{array}\right]
      \end{matrix}
    \right) =&
    \left(
      \begin{matrix}
        \left[\begin{array}{l}
                \vct{s}_i^\alpha \\ 
                \vct{t}_1^\alpha \\
                \vct{\emptyset}(I_2^\alpha,1) \\
                \vct{t}_3^\alpha \\
                \vct{t}_4^\alpha \\
                \vct{t}_5^\alpha \\
                \vct{t}_6^\alpha 
\end{array}\right]
\\
\left[\begin{array}{l}
        \vct{s}_i^\beta \\ 
        \vct{\emptyset}(I_1^\beta,1) \\
        \vct{t}_2^\beta \\
        \vct{t}_3^\beta \\
        \vct{t}_4^\beta \\
        \vct{t}_5^\beta \\
        \vct{t}_6^\beta 
      \end{array}\right]
      \end{matrix}
    \right)
  \end{aligned} \label{eq:2x2}
\end{align}
where $\mtx{\emptyset}$ denotes the zero matrix of size $n\times n$,
and {\bf $I_1^\beta$ and $I_2^\alpha$, indicate indices of the same
points in space, indexed from $\beta$ and $\alpha$ respectively.}
Vectors $\vct{t}$ contain given boundary conditions on $\partial
\Omega$.  

\section{The global system} \label{sec:sparse}

This section presents the construction and application of the large
sparse linear system that results from the HPS discretization.  The
section begins with a high level view of the global system and then
provides the details of rapidly applying it to vector.

\subsection{A high level view}

The linear system that results from the discretization with $N_\ell$
leaves is
\begin{equation}
  \boldsymbol{A} \boldsymbol{x} = \boldsymbol{b} 
\end{equation}
where $\boldsymbol{A}$ is an $\left(N \times N\right)$ non-symmetric
matrix where $N=n N_\ell$ with complex-valued entries and the right
hand side vector $\boldsymbol{b}$ has entries that correspond to
$s(x)$ in (\ref{eqn:diffhelm}) evaluated at the interior nodes, $t(x)$
in (\ref{eqn:diffhelm}) evaluated at the boundary nodes and zeros for
the inter-leaf boundary nodes.  The zeros in $\vct{b}$ correspond to
the equations that enforce the continuity of impedance boundary data
between leaf boxes.  The matrix $\mtx{A}$ is sparse with the spectral
collocation matrices $\mtx{A}^\tau$ defined in \eqref{eqn:leaf} along
the diagonal and sparse matrices off the diagonal for enforcing the
continuity of impedance between leaf boxes.  The off-diagonal non-zero
matrices enforce the continuity of impedance data between leaf boxes.
This matches the two box linear system in equation (\ref{eq:2x2}).  
 Figure
\ref{fig:GlobalSystemCorner}(b) illustrates the block-sparsity pattern
of the large system and Figure \ref{fig:GlobalSystemCorner}(c)
illustrates a principal block-submatrix of this sparse system for a 
mesh defined in Figure \ref{fig:GlobalSystemCorner}(a).   
\begin{figure}
  \begin{subfigure}{0.5\textwidth}
    \begin{tikzpicture}[scale=1.7]

      \def \dx {0.25}
      \def \dy {0.25}
      \def \dz {0.25}
      \def \opacity {0.8}

      \foreach \z in {0.75,1,...,1.5}
      \foreach \x in {0,0.25,...,0.75}
      \foreach \y in {0,0.25,...,0.75}
      {
        \filldraw[fill=blue!20,fill opacity=\opacity] ({0+0.5*(1-\dy)+\x-0.5*\y},{0+0.5*(1-\dy)-0.5*\y+\z}) -- ({0+0.5+\x-0.5*\y},{0+0.5-0.5*\y+\z}) -- ({0+0.5+\x-0.5*\y},{0+0.5+\dz-0.5*\y+\z}) -- ({0+0.5*(1-\dy)+\x-0.5*\y},{0+\dz+0.5*(1-\dy)-0.5*\y+\z}) -- ({0+0.5*(1-\dy)+\x-0.5*\y},{0+0.5*(1-\dy)-0.5*\y+\z}); 
        \filldraw[fill=blue!20,fill opacity=\opacity] ({0.5+0+\x-0.5*\y},{0.5+0-0.5*\y+\z}) -- ({0.5+\dx+\x-0.5*\y},{0.5+0-0.5*\y+\z}) -- ({0.5+\dx+\x-0.5*\y},{0.5+\dz-0.5*\y+\z})  -- ({0.5+0+\x-0.5*\y},{0.5+\dz-0.5*\y+\z}) -- ({0.5+0+\x-0.5*\y},{0.5+0-0.5*\y+\z});
        \filldraw[fill=blue!20,fill opacity=\opacity] ({0+0.5*(1-\dy)+\x-0.5*\y},{0+0.5*(1-\dy)-0.5*\y+\z}) -- ({0+\dx+0.5*(1-\dy)+\x-0.5*\y},{0+0.5*(1-\dy)-0.5*\y+\z}) -- ({0+\dx+0.5+\x-0.5*\y},{0+0.5-0.5*\y+\z}) -- ({0+0.5+\x-0.5*\y},{0+0.5-0.5*\y+\z}) -- ({0+0.5*(1-\dy)+\x-0.5*\y},{0+0.5*(1-\dy)-0.5*\y+\z}); 

        \filldraw[fill=blue!20,fill opacity=\opacity] ({0+0.5*(1-\dy)+\dx+\x-0.5*\y},{0+0.5*(1-\dy)-0.5*\y+\z}) -- ({0+0.5+\dx+\x-0.5*\y},{0+0.5-0.5*\y+\z}) -- ({0+0.5+\dx+\x-0.5*\y},{0+0.5+\dz-0.5*\y+\z}) -- ({0+0.5*(1-\dy)+\dx+\x-0.5*\y},{0+\dz+0.5*(1-\dy)-0.5*\y+\z}) -- ({0+0.5*(1-\dy)+\dx+\x-0.5*\y},{0+0.5*(1-\dy)-0.5*\y+\z}) ;
        \filldraw[fill=blue!20,fill opacity=\opacity] ({0.5+0-0.5*\dy+\x-0.5*\y},{0.5+0-0.5*\dy-0.5*\y+\z}) -- ({0.5+\dx-0.5*\dy+\x-0.5*\y},{0.5+0-0.5*\dy-0.5*\y+\z}) -- ({0.5+\dx-0.5*\dy+\x-0.5*\y},{0.5+\dz-0.5*\dy-0.5*\y+\z})  -- ({0.5+0-0.5*\dy+\x-0.5*\y},{0.5+\dz-0.5*\dy-0.5*\y+\z}) -- ({0.5+0-0.5*\dy+\x-0.5*\y},{0.5+0-0.5*\dy-0.5*\y+\z}) ;
        \filldraw[fill=blue!20,fill opacity=\opacity] ({0+0.5*(1-\dy)+\x-0.5*\y},{0+\dz+0.5*(1-\dy)-0.5*\y+\z}) -- ({0+\dx+0.5*(1-\dy)+\x-0.5*\y},{0+\dz+0.5*(1-\dy)-0.5*\y+\z}) -- ({0+\dx+0.5+\x-0.5*\y},{0+0.5+\dz-0.5*\y+\z}) -- ({0+0.5+\x-0.5*\y},{0+0.5+\dz-0.5*\y+\z}) -- ({0+0.5*(1-\dy)+\x-0.5*\y},{0+\dz+0.5*(1-\dy)-0.5*\y+\z});
      }

      \draw[-stealth] (1.55,2.125) -- (1.95,2.4);
      \draw[-stealth] (1.55,1.85) -- (1.95,1.85);
      \draw[-stealth] (1.55,1.625) -- (1.95,1.25);
      \draw[-stealth] (1.55,1.325) -- (1.95,0.625);
      
      \pgfmathsetmacro{\w}{0}
      \foreach \z in {0,0.75,...,2.25}
      \foreach \x in {2,2.25,...,2.75}
      \foreach \y in {0,0.25,...,0.75}
      {
        \pgfmathsetmacro{\ww}{int(\w+1)}
        
        \filldraw[fill=blue!20,fill opacity=\opacity] ({0+0.5*(1-\dy)+\x-0.5*\y},{0+0.5*(1-\dy)-0.5*\y+\z}) -- ({0+0.5+\x-0.5*\y},{0+0.5-0.5*\y+\z}) -- ({0+0.5+\x-0.5*\y},{0+0.5+\dz-0.5*\y+\z}) -- ({0+0.5*(1-\dy)+\x-0.5*\y},{0+\dz+0.5*(1-\dy)-0.5*\y+\z}) -- ({0+0.5*(1-\dy)+\x-0.5*\y},{0+0.5*(1-\dy)-0.5*\y+\z}); 
        \filldraw[fill=blue!20,fill opacity=\opacity] ({0.5+0+\x-0.5*\y},{0.5+0-0.5*\y+\z}) -- ({0.5+\dx+\x-0.5*\y},{0.5+0-0.5*\y+\z}) -- ({0.5+\dx+\x-0.5*\y},{0.5+\dz-0.5*\y+\z})  -- ({0.5+0+\x-0.5*\y},{0.5+\dz-0.5*\y+\z}) -- ({0.5+0+\x-0.5*\y},{0.5+0-0.5*\y+\z});
        \filldraw[fill=blue!20,fill opacity=\opacity] ({0+0.5*(1-\dy)+\x-0.5*\y},{0+0.5*(1-\dy)-0.5*\y+\z}) -- ({0+\dx+0.5*(1-\dy)+\x-0.5*\y},{0+0.5*(1-\dy)-0.5*\y+\z}) -- ({0+\dx+0.5+\x-0.5*\y},{0+0.5-0.5*\y+\z}) -- ({0+0.5+\x-0.5*\y},{0+0.5-0.5*\y+\z}) -- ({0+0.5*(1-\dy)+\x-0.5*\y},{0+0.5*(1-\dy)-0.5*\y+\z}); 

        \filldraw[fill=blue!20,fill opacity=\opacity] ({0+0.5*(1-\dy)+\dx+\x-0.5*\y},{0+0.5*(1-\dy)-0.5*\y+\z}) -- ({0+0.5+\dx+\x-0.5*\y},{0+0.5-0.5*\y+\z}) -- ({0+0.5+\dx+\x-0.5*\y},{0+0.5+\dz-0.5*\y+\z}) -- ({0+0.5*(1-\dy)+\dx+\x-0.5*\y},{0+\dz+0.5*(1-\dy)-0.5*\y+\z}) -- ({0+0.5*(1-\dy)+\dx+\x-0.5*\y},{0+0.5*(1-\dy)-0.5*\y+\z}) ;
        \filldraw[fill=blue!20,fill opacity=\opacity] ({0.5+0-0.5*\dy+\x-0.5*\y},{0.5+0-0.5*\dy-0.5*\y+\z}) -- ({0.5+\dx-0.5*\dy+\x-0.5*\y},{0.5+0-0.5*\dy-0.5*\y+\z}) -- ({0.5+\dx-0.5*\dy+\x-0.5*\y},{0.5+\dz-0.5*\dy-0.5*\y+\z})  -- ({0.5+0-0.5*\dy+\x-0.5*\y},{0.5+\dz-0.5*\dy-0.5*\y+\z}) -- ({0.5+0-0.5*\dy+\x-0.5*\y},{0.5+0-0.5*\dy-0.5*\y+\z}) ;
        \filldraw[fill=blue!20,fill opacity=\opacity] ({0+0.5*(1-\dy)+\x-0.5*\y},{0+\dz+0.5*(1-\dy)-0.5*\y+\z}) -- ({0+\dx+0.5*(1-\dy)+\x-0.5*\y},{0+\dz+0.5*(1-\dy)-0.5*\y+\z}) -- ({0+\dx+0.5+\x-0.5*\y},{0+0.5+\dz-0.5*\y+\z}) -- ({0+0.5+\x-0.5*\y},{0+0.5+\dz-0.5*\y+\z}) -- ({0+0.5*(1-\dy)+\x-0.5*\y},{0+\dz+0.5*(1-\dy)-0.5*\y+\z});
        \node[label={[label distance=0mm]:$\scalemath{0.6}{\ww}$}] at ({0+0.5*(1-\dy)+\x-0.5*\y+0.2},{0+\dz+0.5*(1-\dy)-0.5*\y+\z-0.125}){};
        \global\let\w=\ww+1
      }
\end{tikzpicture}
  \caption{}
\end{subfigure}
\begin{subfigure}{0.49\textwidth}
  \centering
  \[\begin{pmatrix}
      \includegraphics[width=6cm]{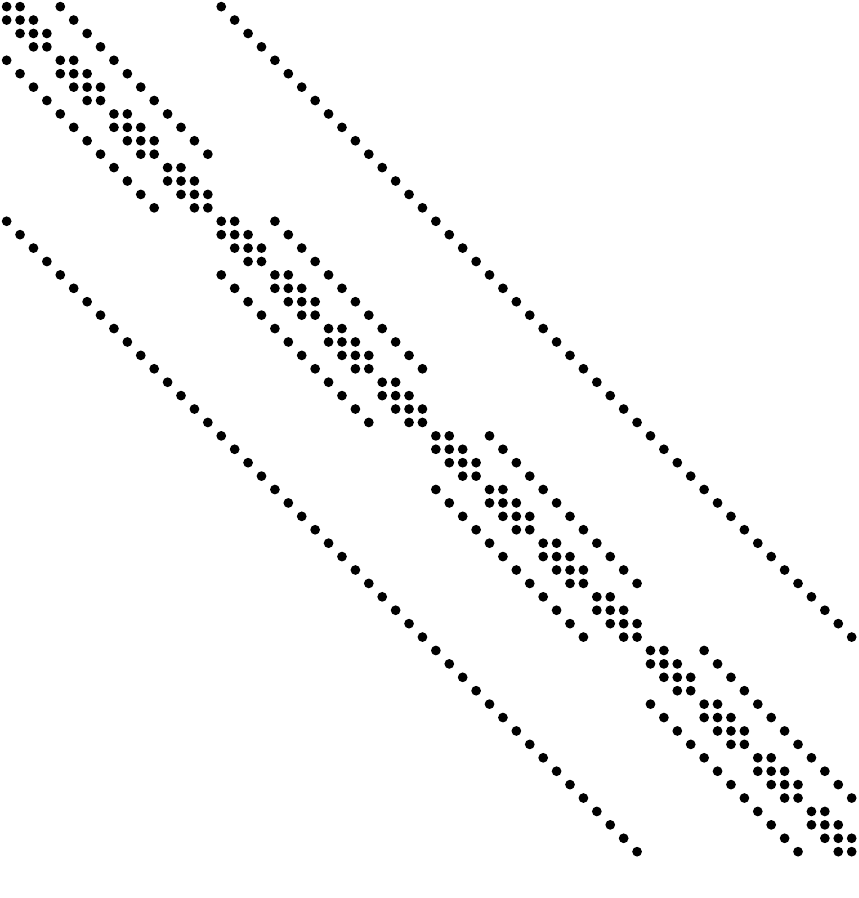}
    \end{pmatrix}\]
  \caption{}
\end{subfigure}
\begin{subfigure}{\textwidth}
\centering    \begin{equation*}
      \scalemath{0.58}{
        \left(\begin{array}{cccccccccccccccccc}
                \boldsymbol{A}^{\tau=1}       & \boldsymbol{G}^{\tau=2}_{-x} &                               &                               & \boldsymbol{G}^{\tau=5}_{-y} &                               &                               &                               &                               &                                &                                &                                &                                &                                &                                &                                & \boldsymbol{G}^{\tau=17}_{-z} \\
                \boldsymbol{G}^{\tau=1}_{x}  &  \boldsymbol{A}^{\tau=2}      & \boldsymbol{G}^{\tau=3}_{-x} &                               &                               & \boldsymbol{G}^{\tau=6}_{-y} &                               &                               &                               &                                &                                &                                &                                &                                &                                &                                &                                \\
                                              & \boldsymbol{G}^{\tau=2}_{x}  &  \boldsymbol{A}^{\tau=3}      & \boldsymbol{G}^{\tau=4}_{-x} &                               &                               & \boldsymbol{G}^{\tau=7}_{-y} &                               &                               &                                &                                &                                &                                &                                &                                &                                &                                \\
                                              &                               & \boldsymbol{G}^{\tau=3}_{x}  &  \boldsymbol{A}^{\tau=4}      &                               &                               &                               & \boldsymbol{G}^{\tau=8}_{-y} &                               &                                &                                &                                &                                &                                &                                &                                &                                \\
                \boldsymbol{G}^{\tau=1}_{y}  &                               &                               &                               &  \boldsymbol{A}^{\tau=5}      & \boldsymbol{G}^{\tau=6}_{-x} &                               &                               & \boldsymbol{G}^{\tau=9}_{-y} &                                &                                &                                &                                &                                &                                &                                &                                \\
                                              & \boldsymbol{G}^{\tau=2}_{y}  &                               &                               & \boldsymbol{G}^{\tau=5}_{x}  &  \boldsymbol{A}^{\tau=6}      & \boldsymbol{G}^{\tau=7}_{-x} &                               &                               & \boldsymbol{G}^{\tau=10}_{-y} &                                &                                &                                &                                &                                &                                &                                \\
                                              &                               & \boldsymbol{G}^{\tau=3}_{y}  &                               &                               & \boldsymbol{G}^{\tau=6}_{x}  &  \boldsymbol{A}^{\tau=7}      & \boldsymbol{G}^{\tau=8}_{-x} &                               &                                & \boldsymbol{G}^{\tau=11}_{-y} &                                &                                &                                &                                &                                &                                \\
                                              &                               &                               & \boldsymbol{G}^{\tau=4}_{y}  &                               &                               & \boldsymbol{G}^{\tau=7}_{x}  &  \boldsymbol{A}^{\tau=8}      &                               &                                &                                & \boldsymbol{G}^{\tau=12}_{-y} &                                &                                &                                &                                &                                \\
                                              &                               &                               &                               & \boldsymbol{G}^{\tau=5}_{y}  &                               &                               &                               &  \boldsymbol{A}^{\tau=9}      & \boldsymbol{G}^{\tau=10}_{-x} &                                &                                & \boldsymbol{G}^{\tau=13}_{-y} &                                &                                &                                &                                \\
                                              &                               &                               &                               &                               & \boldsymbol{G}^{\tau=6}_{y}  &                               &                               & \boldsymbol{G}^{\tau=9}_{x}  &  \boldsymbol{A}^{\tau=10}      & \boldsymbol{G}^{\tau=11}_{-x} &                                &                                & \boldsymbol{G}^{\tau=14}_{-y} &                                &                                &                                \\
                                              &                               &                               &                               &                               &                               & \boldsymbol{G}^{\tau=7}_{y}  &                               &                               & \boldsymbol{G}^{\tau=10}_{x}  &  \boldsymbol{A}^{\tau=11}      & \boldsymbol{G}^{\tau=12}_{-x} &                                &                                & \boldsymbol{G}^{\tau=15}_{-y} &                                &                                \\
                                              &                               &                               &                               &                               &                               &                               & \boldsymbol{G}^{\tau=8}_{y}  &                               &                                & \boldsymbol{G}^{\tau=11}_{x}  &  \boldsymbol{A}^{\tau=12}      &                                &                                &                                & \boldsymbol{G}^{\tau=16}_{-y} &                                \\
                                              &                               &                               &                               &                               &                               &                               &                               & \boldsymbol{G}^{\tau=9}_{y}  &                                &                                &                                &  \boldsymbol{A}^{\tau=13}      & \boldsymbol{G}^{\tau=14}_{-x} &                                &                                &                                \\
                                              &                               &                               &                               &                               &                               &                               &                               &                               & \boldsymbol{G}^{\tau=10}_{y}  &                                &                                & \boldsymbol{G}^{\tau=13}_{x}  &  \boldsymbol{A}^{\tau=14}      & \boldsymbol{G}^{\tau=15}_{-x} &                                &                                \\
                                              &                               &                               &                               &                               &                               &                               &                               &                               &                                & \boldsymbol{G}^{\tau=11}_{y}  &                                &                                & \boldsymbol{G}^{\tau=14}_{x}  &  \boldsymbol{A}^{\tau=15}      & \boldsymbol{G}^{\tau=16}_{-x} &                                \\
                                              &                               &                               &                               &                               &                               &                               &                               &                               &                                &                                & \boldsymbol{G}^{\tau=12}_{y}  &                                &                                & \boldsymbol{G}^{\tau=15}_{x}  &  \boldsymbol{A}^{\tau=16}      &                                \\
                \boldsymbol{G}^{\tau=1}_{z}  &                               &                               &                               &                               &                               &                               &                               &                               &                                &                                &                                &                                &                                &                                &                                & \boldsymbol{A}^{\tau=17}       \\
              \end{array}\right)}
        \end{equation*} \caption{}
\end{subfigure}
\caption{
  (a) Numbering of leaves in a uniform mesh with $4$ leaf boxes in 
  each direction.\\
  (b) Illustration of the block sparsity structure for a $4\times
  4\times 4$ leaves uniform mesh, when using the numerotation in (a).
  Each dot in this figure is a block matrix, diagonal blocks are leaf
  matrices $\boldsymbol{A}^{\tau}$, off-diagonal blocks are outgoing
  impedance matrices $\boldsymbol{G}^\tau_{-x}$,
  $\boldsymbol{G}^\tau_{x}$, $\boldsymbol{G}^\tau_{-y}$,
  $\boldsymbol{G}^\tau_{y}$, $\boldsymbol{G}^\tau_{-z}$ or
  $\boldsymbol{G}^\tau_{z}$, and white space are zero blocks. \\
  (c) Zoom on the $17\times 17$-block upper-left corner of the
  block matrix depicted in (b).}
\label{fig:GlobalSystemCorner}
\end{figure}

\clearpage
In this illustration of the full system, the $\mtx{G}^\tau$ with a
subscript matrices are the matrices which help enforce the continuity
of the impedance data between leaf boxes.  They are sparse matrices
defined as follows:
\begin{align}
  \label{eqn:outgoingimpedance}
  \begin{aligned}
    \mtx{G}^\tau_{-x} :=& 
    \left(\begin{matrix}
        \boldsymbol{\emptyset}\left(I^{\tau'}_i,I^\tau\right) \\
        \boldsymbol{\emptyset}\left(I^{\tau'}_1,I^\tau\right) \\
        \mtx{G}^\tau\left(I^{\tau}_1,I^\tau\right) \\
        \boldsymbol{\emptyset}\left(I^{\tau'}_3,I^\tau\right) \\
        \boldsymbol{\emptyset}\left(I^{\tau'}_4,I^\tau\right) \\
        \boldsymbol{\emptyset}\left(I^{\tau'}_5,I^\tau\right) \\
        \boldsymbol{\emptyset}\left(I^{\tau'}_6,I^\tau\right)
      \end{matrix}\right),
    &
    \mtx{G}^\tau_{x} :=&
    \left(\begin{matrix}
        \boldsymbol{\emptyset}\left(I^{\tau'}_i,I^\tau\right) \\
        \mtx{G}^\tau\left(I^{\tau}_2,I^\tau\right) \\
        \boldsymbol{\emptyset}\left(I^{\tau'}_2,I^\tau\right) \\
        \boldsymbol{\emptyset}\left(I^{\tau'}_3,I^\tau\right) \\
        \boldsymbol{\emptyset}\left(I^{\tau'}_4,I^\tau\right) \\
        \boldsymbol{\emptyset}\left(I^{\tau'}_5,I^\tau\right) \\
        \boldsymbol{\emptyset}\left(I^{\tau'}_6,I^\tau\right)
      \end{matrix}\right),
    &
    \mtx{G}^\tau_{-y} :=& 
    \left(\begin{matrix}
        \boldsymbol{\emptyset}\left(I^{\tau'}_i,I^\tau\right) \\
        \boldsymbol{\emptyset}\left(I^{\tau'}_1,I^\tau\right) \\
        \boldsymbol{\emptyset}\left(I^{\tau'}_2,I^\tau\right) \\
        \boldsymbol{\emptyset}\left(I^{\tau'}_3,I^\tau\right) \\
        \mtx{G}^\tau\left(I^{\tau}_3,I^\tau\right) \\
        \boldsymbol{\emptyset}\left(I^{\tau'}_5,I^\tau\right) \\
        \boldsymbol{\emptyset}\left(I^{\tau'}_6,I^\tau\right)
      \end{matrix}\right),
    \\
    \mtx{G}^\tau_{y} :=&
    \left(\begin{matrix}
        \boldsymbol{\emptyset}\left(I^{\tau'}_i,I^\tau\right) \\
        \boldsymbol{\emptyset}\left(I^{\tau'}_1,I^\tau\right) \\
        \boldsymbol{\emptyset}\left(I^{\tau'}_2,I^\tau\right) \\
        \mtx{G}^\tau\left(I^{\tau}_4,I^\tau\right) \\
        \boldsymbol{\emptyset}\left(I^{\tau'}_4,I^\tau\right) \\
        \boldsymbol{\emptyset}\left(I^{\tau'}_5,I^\tau\right) \\
        \boldsymbol{\emptyset}\left(I^{\tau'}_6,I^\tau\right)
      \end{matrix}\right),
    &
    \mtx{G}^\tau_{-z} :=& 
    \left(\begin{matrix}
        \boldsymbol{\emptyset}\left(I^{\tau'}_i,I^\tau\right) \\
        \boldsymbol{\emptyset}\left(I^{\tau'}_1,I^\tau\right) \\
        \boldsymbol{\emptyset}\left(I^{\tau'}_2,I^\tau\right) \\
        \boldsymbol{\emptyset}\left(I^{\tau'}_3,I^\tau\right) \\
        \boldsymbol{\emptyset}\left(I^{\tau'}_4,I^\tau\right) \\
        \boldsymbol{\emptyset}\left(I^{\tau'}_5,I^\tau\right) \\
        \mtx{G}^\tau\left(I^{\tau}_5,I^\tau\right) \\
      \end{matrix}\right),
    &
    \mtx{G}^\tau_{z} :=&
    \left(\begin{matrix}
        \boldsymbol{\emptyset}\left(I^{\tau'}_i,I^\tau\right) \\
        \boldsymbol{\emptyset}\left(I^{\tau'}_1,I^\tau\right) \\
        \boldsymbol{\emptyset}\left(I^{\tau'}_2,I^\tau\right) \\
        \boldsymbol{\emptyset}\left(I^{\tau'}_3,I^\tau\right) \\
        \boldsymbol{\emptyset}\left(I^{\tau'}_4,I^\tau\right) \\
        \mtx{G}^\tau\left(I^{\tau}_6,I^\tau\right) \\
        \boldsymbol{\emptyset}\left(I^{\tau'}_6,I^\tau\right) \\
      \end{matrix}\right).
  \end{aligned}
\end{align}
where $\tau'$ indicates the corresponding neighboring leaf to $\tau$,
$\mtx{\emptyset}$ denotes the zero matrix of size $n\times n$ and the
non-zero matrices are submatrices of $\mtx{G}^\tau$ as defined in
equation \eqref{eqn:G}.  

Refering back the two leaf example in the previous section, the off-diagonal
blocks are $\mtx{G}_{-x}^\beta$ and $\mtx{G}_{x}^\alpha$ in the $(1,2)$ and $(2,1)$
positions respectively.  

For mesh involving purely interior leaf boxes, all of the
$\mtx{G}^\tau$ with a subscript matrices are necessary.  For example,
consider the box $\tau = 43$ in Figure
\ref{fig:GlobalSystemCorner}(a).  Then the matrices
$\mtx{G}^{\tau=44}_{-x}$, $\mtx{G}^{\tau=42}_{x}$, $\mtx{G}^{\tau=47}_{-y}$,
$\mtx{G}^{\tau=39}_{y}$, $\mtx{G}^{\tau=59}_{-z}$ and $\mtx{G}^{\tau=27}_{z}$ are
needed to enforce the continuity of the fluxes through all
\emph{six} of the interfaces.

 \begin{remark} 
To keep computations local to each leaf, we allow each leaf box to
have its own set of boundary nodes. This means that all interfaces of
leaf boxes that are not on the boundary of $\Omega$ have twice as many
unknowns as there are discretization points on that face.
\end{remark}

\subsection{The application of the linear system}
The solution technique presented in this paper utilizes a
\texttt{GMRES} iterative solver.  In order for this solver to be
efficient, the matrix $\mtx{A}$ be able to applied to a vector
rapidly.  Since all of the submatrices of $\mtx{A}$ are defined at the
box (element) level, it is easy to apply $\mtx{A}$ in a
\emph{matrix-free} manner \cite{Knyazev2001, KronbichlerKormann2019}.
Specifically, for any block row corresponding to element $\tau$, there
are two matrix types that need to be applied: the self interaction
matrix $\boldsymbol{A}^{\tau}$ corresponding to the discretization of
the problem on the element as presented in section
\ref{sec:leafdiscretization} and $\mtx{G}^\tau$ with subscript matrices
which enforce the continuity of impedance boundary data.  Fortunately both
of these matrices are sparse.  Figure \ref{fig:sparse_loc} illustrates the sparsity pattern of the
matrix $\mtx{A}^\tau$ and $\mtx{G}^\tau_{-x}$ when $n_c = 10$.  Since
the $\mtx{G}^\tau_x$, etc. matrices are very sparse (e.g. Figure
\ref{fig:sparse_loc}(b)), their application to a vector is
straightforward.

\begin{figure}
  \centering
  \begin{subfigure}{0.45\textwidth}
    \centering
    \[\begin{pmatrix}
      \includegraphics[width=0.9\textwidth]{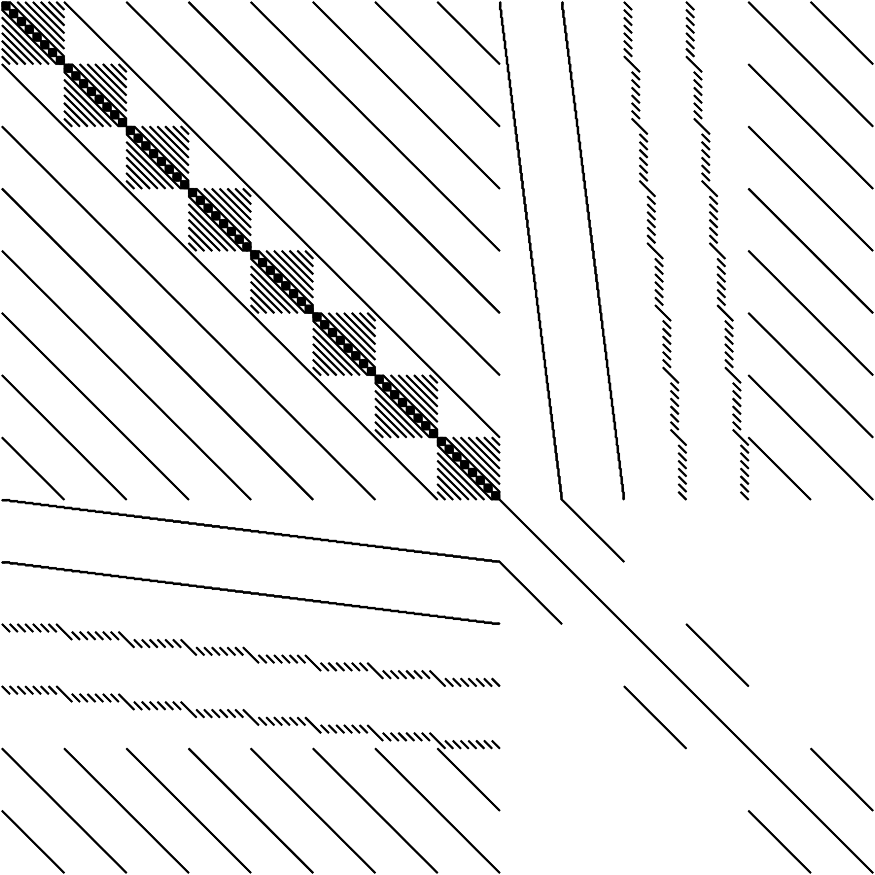}
    \end{pmatrix}\]
    \caption{Sparsity of a leaf block $\mtx{A}^\tau$}
  \end{subfigure}
  {\mbox{\hspace{0.2cm}}}
  \begin{subfigure}{0.45\textwidth} 
    \[\begin{pmatrix}
        \includegraphics[width=0.9\textwidth]{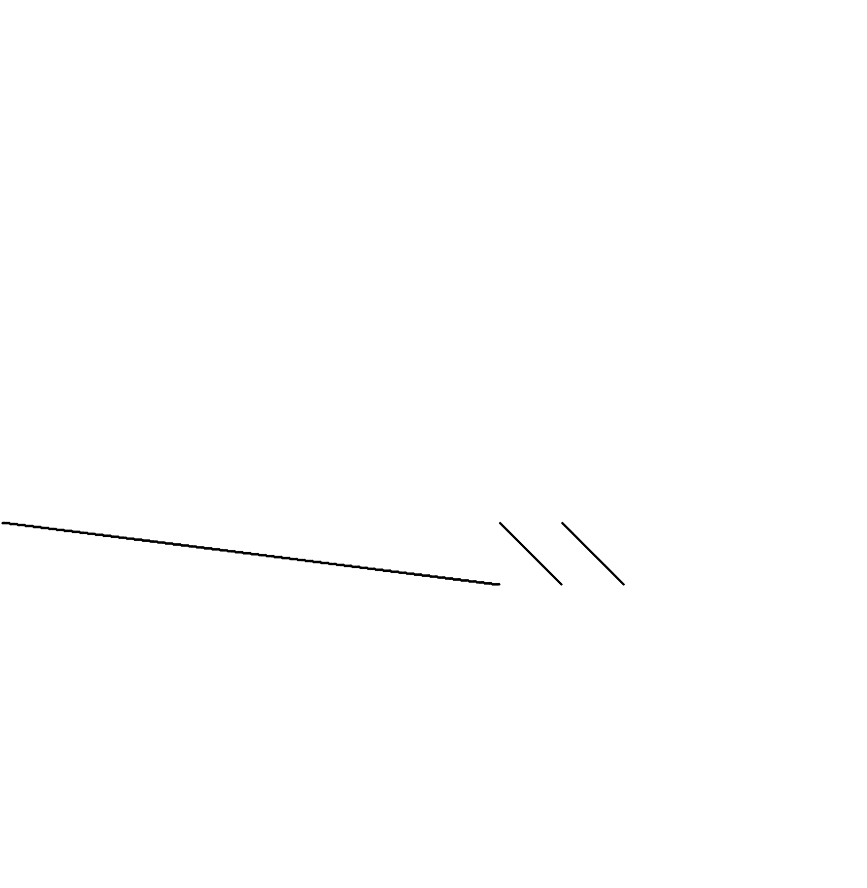}
    \end{pmatrix}\]
    \caption{Sparsity of an impedance block $\mtx{G}^\tau_{x}$}
  \end{subfigure}
  \caption{Illustration of sparsity pattern of (a) a leaf block
    $\mtx{A}^\tau$ and (b) an outgoing impedance block
    $\mtx{G}^\tau_{x}$ when the discretization technique is applied to
    leaf boxes with $n_c = 10$. In these figures black dots represent
    non-zero elements of the matrix, and empty white space represent zero
    elements.
    \label{fig:sparse_loc}}
\end{figure}

The matrix $\mtx{A}^\tau$ is the largest and most dense matrix in a
block row of $\mtx{A}$.  The most dense subblock of $\mtx{A}^\tau$ is
the principal $n_i \times n_i$ submatrix $\mtx{A}^\tau_{ii}$ from
(\ref{eqn:leaf}).  Recall that the submatrix $\mtx{A}^\tau_{ii}$
corresponds to the spectral approximation of the differential operator
on the interior nodes.  Each of the derivative operators can be
written as a collection of Kronecker products involving identity
matrices and the one dimensional second derivative operator.  Thus
$\mtx{A}_{ii}^\tau$ can be expressed almost exclusively in terms of
Kronecker products.  Specifically,
\begin{equation}\label{eqn:leafmat}
  \boldsymbol{A}_{ii}^\tau =
  - \boldsymbol{I}_{(n_c-2)} \otimes \boldsymbol{I}_{(n_c-2)} \otimes \boldsymbol{L}_1
  - \boldsymbol{I}_{(n_c-2)} \otimes \boldsymbol{L}_1 \otimes \boldsymbol{I}_{(n_c-2)}
  - \boldsymbol{L}_1 \otimes \boldsymbol{I}_{(n_c-2)} \otimes \boldsymbol{I}_{(n_c-2)}
  - \boldsymbol{C}_{ii}^\tau
\end{equation}
where $\mtx{I}_{(n_c-2)}$ denotes the $(n_c-2)\times (n_c-2)$ identity
matrix, $\boldsymbol{L}_1$ denotes the $(n_c-2) \times (n_c-2)$
submatrix of the Chebyshev differentiation matrix used for
approximating the second derivative of a one dimensional function, and
$\boldsymbol{C}_{ii}^\tau$ denotes the operator $\boldsymbol{C}^\tau$
on the interior nodes.  Recall that $\mtx{C}^\tau$ is a diagonal
matrix.  This structure allows for $\mtx{A}_{ii}^\tau$ to be applied
rapidly via the technique in Algorithm 1.  The algorithm makes use of
the fast application of tensor products presented in Lemma
\ref{lem:tensor} \cite{Roth1934}.
\begin{lemma} \label{lem:tensor}
  Let $\boldsymbol{x} = \text{vec}(\boldsymbol{X})$ denote the
  \emph{vectorization} of the $m\times m$ matrix $\boldsymbol{X}$ formed
  by stacking the columns of $\boldsymbol{X}$ into a single column
  vector $\boldsymbol{x}$. Likewise, let $\boldsymbol{y}$ denote the
  vectorization of the $m\times m$ matrix $\boldsymbol{Y}$.  Let
  $\mtx{M}$ and $\mtx{N}$ denote matrices of size $m \times m$.  The
  Kronecker product matrix vector multiplication
  \[\boldsymbol{y} = (\boldsymbol{M} \otimes \boldsymbol{N})
    \boldsymbol{x}\]
  can be evaluated by creating the vectorization of the following
  \[\boldsymbol{Y} = \boldsymbol{N} \boldsymbol{X}
    \boldsymbol{M}^\intercal.\]
  In other words, $\boldsymbol{y} = \text{vec}\left(\boldsymbol{N}
    \boldsymbol{X} \boldsymbol{M}^\intercal\right)$.
\end{lemma}
\begin{proof}
  See \cite{Roth1934}.
\end{proof}
Effectively, this lemma allows for
the tensor products to be reduced to the cost of applying the one
dimensional derivative matrix.  Algorithm 2 presents the efficient
technique for applying $\mtx{A}$ to a vector.
\begin{FloatingAlgorithm}
\noindent\fbox{
\begin{minipage}{\textwidth}
\begin{center}
\textsc{Algorithm 1} (Fast application of $\boldsymbol{A}^\tau_{ii}$)
\end{center}
Let $\boldsymbol{v}$ be a vector of size $n_i$, the algorithm
calculates $\boldsymbol{w} = \boldsymbol{A}^\tau_{ii} \boldsymbol{v}$.
Let $n_1:=n_c-2$.
\\
The function reshape$(\cdot,(p,q))$ is the classical reshape by
columns function, as found in \texttt{MATLAB} and \texttt{FORTRAN}
implementations.\\
\rule{\textwidth}{0.5pt}
\begin{tabbing}
\mbox{}\hspace{7mm} \= \mbox{}\hspace{6mm} \= \mbox{}\hspace{6mm} \=
\mbox{}\hspace{6mm} \= \mbox{}\hspace{6mm} \= \kill
(1) \> $\boldsymbol{M}_1 =
\text{reshape}\left(\boldsymbol{v},\left(n_1^2,n_1\right)\right))$ \\
(2) \> $\boldsymbol{w}_1 =
\text{vec}\left(\boldsymbol{M}_1
  \boldsymbol{L}_1^\intercal\right)$ \\
(3) \> \textbf{for} $j = 1, \dots, n_1$ \\
(4) \>\> $\boldsymbol{M}_2(:,j) =
\text{vec}\left(\text{reshape}(\boldsymbol{M}_1(:,j),(n_1,n_1))
  \boldsymbol{L}_1^\intercal\right)$ \\
(5) \> \textbf{end do} \\
(6) \> $\boldsymbol{w}_2 = \text{vec}\left(\boldsymbol{M}_2\right)$ \\
(7) \> \textbf{for} $j = 1, \dots, n_1$ \\
(8) \>\> $\boldsymbol{M}_2(:,j) = \text{vec}\left(\boldsymbol{L}_1
  \text{reshape}\left(\boldsymbol{M}_1(:,j),\left(n_1,n_1\right)\right)\right)$
\\
(9) \> \textbf{end for} \\
(10)\> $\boldsymbol{w}_3 = \text{vec}\left(\boldsymbol{M}_2\right)$ \\
(11)\> $\boldsymbol{w} = -\boldsymbol{w}_1 - \boldsymbol{w}_2 -
\boldsymbol{w}_3 + \boldsymbol{C}^\tau_{ii} \boldsymbol{v}$
\end{tabbing}
\end{minipage}}
\end{FloatingAlgorithm}

\begin{FloatingAlgorithm}
\noindent\fbox{
\begin{minipage}{\textwidth}
\begin{center}
\textsc{Algorithm 2} $\left(\text{Application of the forward operator }
\mtx{A}\right)$
\end{center}
Let \[\boldsymbol{v} = \left(\begin{smallmatrix}
\boldsymbol{v}^{\tau=1} \\ \dots \\
\boldsymbol{v}^{\tau=N_\ell}\end{smallmatrix} \right)\] be a vector of
size $N$, where \[\boldsymbol{v}^\tau =
\left(\begin{smallmatrix} \boldsymbol{v}^\tau_i \\
\boldsymbol{v}^\tau_b \end{smallmatrix}\right)\] is the sub-vector
of size $n=n_i+n_b$ associated to leaf $\tau$, where
$\boldsymbol{v}^\tau_i$ is of size $n_i$ and $\boldsymbol{v}^\tau_b$
is of size $n_b$.  The vectors $\boldsymbol{w}$ and
$\boldsymbol{w}^\tau$ are defined in the same manner.
The algorithm calculates $\boldsymbol{w} =
\mtx{A} \boldsymbol{v}$.\\
\rule{\textwidth}{0.5pt}
\begin{tabbing}
\mbox{}\hspace{7mm} \= \mbox{}\hspace{6mm} \= \mbox{}\hspace{6mm} \=
\mbox{}\hspace{6mm} \= \mbox{}\hspace{6mm} \= \kill
(1) \> \textbf{for} $\tau = 1, \dots, N_\ell$ \\
(2) \> \> calculate $\vct{w}^\tau_i = \mtx{A}_{ii}^\tau \vct{v}^\tau_i
+ \mtx{A}^\tau(I_i,I_b) \vct{v}^\tau_b$ using \eqref{eqn:leafmat},
Lemma \ref{lem:tensor} and \\
\> \> \> sparse $\mtx{A}^\tau(I_i,I_b)$, \\
(3) \> \> calculate $\boldsymbol{w}_b = \mtx{F}(I_b,I_i)
\vct{v}^\tau_i + \mtx{F}(I_b,I_b) \vct{v}^\tau_b$ using sparse
$\mtx{F}(I_b,I_i)$ \\
\> \> \> and $\mtx{F}(I_b,I_i)$,\\
(4) \> \> \textbf{for} $j = -x,x,-y,y,-z,-z$ \textbf{not on } $\partial\Omega$ \\
(5) \> \> \> set $\vct{w}^\tau = \vct{w}^\tau + \mtx{G}^\tau_j
\vct{w}^\tau$, using definitions \eqref{eqn:outgoingimpedance}\\
(6) \> \> \textbf{end for} \\
(7) \> \textbf{end for}
\end{tabbing}
\end{minipage}}
\end{FloatingAlgorithm}

\section{The Preconditioner} \label{sec:pre}
While the sparse linear system that arises from the HPS discretization
can be applied rapidly, its condition number can be quite large.  In
fact, the condition number of the linear system is highly related to
the condition number of the leaf (or element) discretizations. Because
of this dependence, we chose to build a block Jacobi preconditioner
that tackles the condition number of the leaf discretizations.  This
means that the proposed solution technique is to use an iterative
solver to find $\vct{x}$ such that
\begin{equation}
  \mtx{J}^{-1} \mtx{A} \vct{x} = \mtx{J}^{-1} \vct{b}
  \label{eqn:globalsystem}
\end{equation}
where $\mtx{J}^{-1}$ is the block Jacobi preconditioner.  In order for this
to be a viable solution technique, the matrix $\mtx{J}^{-1}$ must be
efficient to construct and apply to a vector.  The block Jacobi preconditioner
is based on using an efficient solver for a homogenized Helmholtz problem 
on each leaf.   

To explain how the preconditioner works, first consider the task of
solving a problem on on leaf.  Recall that the discretized problem on
a leaf takes the form
\begin{equation}\mtx{A}^\tau \vct{u} =
\mtwo{\mtx{A}_{ii}^\tau}{\mtx{A}_{ib}^\tau}{\mtx{F}^\tau_{bi}}{\mtx{F}^\tau_{bb}}
\vtwo{\vct{u}^\tau_i}{\vct{u}^\tau_b} =
\vtwo{\vct{s}^\tau}{\vct{\hat{f}}^\tau}.
\label{eq:leaf}\end{equation}
The solution of (\ref{eq:leaf}) can be expressed via a $2\times 2$
block solve as
\begin{equation}\label{eqn:block}
\begin{aligned}
 \vct{u}_b^\tau & = \mtx{S}^{\tau,-1} \left(\vct{\hat{f}}^\tau -
   \mtx{F}_{bi}^\tau \mtx{A}_{ii}^{\tau,-1} \vct{s}^\tau\right)\\
 \vct{u}_i^\tau & = \mtx{A}_{ii}^{\tau,-1} \left(\vct{s}^\tau -
   \mtx{A}_{ib}^\tau \vct{u}_b^\tau\right)
\end{aligned}.
\end{equation}
where 
\begin{equation}
  \mtx{S}^\tau = \mtx{F}^\tau_{bb} -
  \mtx{F}^\tau_{bi}\mtx{A}_{ii}^{\tau,-1} \mtx{A}^\tau_{ib}
\label{eqn:schur} 
\end{equation}
denotes the Schur complement matrix for leaf $\tau$.  Note that this
requires the inverse of two matrices $\mtx{A}_{ii}^{\tau}$ and
$\mtx{S}^\tau$.

While the matrix $\mtx{A}_{ii}^{\tau}$ is sparse (the principal block
in Figure \ref{fig:sparse_loc}), it is large for high order
discretizations and its inverse is dense.  The Schur complement
$\mtx{S}^\tau$ is dense but its size is on the order of the number of
points on the boundary; i.e. it is much smaller than
$\mtx{A}_{ii}^\tau$.  For example, when $n_c = 16$, the Schur
complement is a matrix of size $1176\times 1176$ while
$\mtx{A}_{ii}^\tau$ is a matrix of size $2744\times
2744$.  Thus $\mtx{S}^\tau$ can be inverted efficiently with standard
linear algebra software such as \texttt{LAPACK}.

With this in mind, we decided to make a preconditioner that
"approximates" $\mtx{A}_{ii}^{\tau,-1}$ but is inexpensive to
construct and apply.  
Let $\tilde{\mtx{A}}^{\tau,-1}_{ii}$ denote the
approximate inverse of $\mtx{A}_{ii}^{\tau}$.  Then the preconditioned
leaf solver is

\begin{equation}\label{eqn:block_pc}
  \begin{aligned}
    \tilde{\vct{u}}_b^\tau & = \tilde{\mtx{S}}^{\tau,-1} \left(\vct{\hat{f}}^\tau -
      \mtx{F}_{bi}^\tau \tilde{\mtx{A}}_{ii}^{\tau,-1} \vct{s}^\tau\right)\\
    \tilde{\vct{u}}_i^\tau & = \tilde{\mtx{A}}_{ii}^{\tau,-1} \left(\vct{s}^\tau -
      \mtx{A}_{ib}^\tau \vct{u}_b^\tau\right)
  \end{aligned}
\end{equation}
where  
\begin{equation}
  \tilde{\mtx{S}}^\tau = \mtx{F}^\tau_{bb} -
  \mtx{F}^\tau_{bi}\tilde{\mtx{A}}_{ii}^{\tau,-1} \mtx{A}^\tau_{ib}
\label{eqn:schur_pc} 
\end{equation}
and $\tilde{\vct{u}}_b^\tau$ and $\tilde{\vct{u}}_i^\tau$ denote the
solutions to this "approximate" problem.  As with the Schur complement
for the original system, $\tilde{\mtx{S}}^{\tau}$ can be inverted
rapidly using standard linear algebra software.

 The preconditioner for the full system is simply applying
(\ref{eqn:block_pc}) as a block diagonal matrix.  Algorithm 4 details
how to efficiently apply the collection of local preconditioners as
the block-Jacobi preconditioner.

Section \ref{sec:hom} provides the details about the matrix $\tilde{\mtx{A}}_{ii}^\tau$
 and the efficient application of its inverse.  Section \ref{sec:homlocal} illustrates
 the performance of the preconditioner when solving a boundary value problem 
 with one leaf box.

\subsection{The matrix $\widetilde{\mtx{A}}_{ii}^\tau$ and its inverse}
\label{sec:hom}
When the local discretization is high order, constructing the inverse
of the matrix $\mtx{A}_{ii}^\tau$ dominates the cost of constructing
the local solution in (\ref{eqn:block}).  This section presents the
choice of the matrix $\tilde{\mtx{A}}^\tau_{ii}$ which makes
(\ref{eqn:block_pc}) an effective local preconditioner and is
inexpensive to invert.

We chose $\tilde{\mtx{A}}^\tau_{ii}$ to be the discretized homogenized
differential operator on the interior of the box $\tau$.  It is known
that the homogenized Helmholtz operator is a good preconditioner for
geometries that are not large in terms of wavelength.  Since in
practice the leaf boxes are never more than 5 wavelengths in size, a
homogenized operator works well as preconditioner for it.  Note that
homogenized operators are, in general, \emph{not good preconditioners}
for high frequency problems since they are spectrally too different
from the non-homogenized counterpart \cite{Sanchez-Palencia1980}.

Specifically, we set
\[\widetilde{\mtx{A}}^\tau_{ii} = - \boldsymbol{I}_{(n_c-2)} \otimes
\boldsymbol{I}_{(n_c-2)} \otimes \boldsymbol{L}_1 -
\boldsymbol{I}_{(n_c-2)} \otimes \boldsymbol{L}_1 \otimes
\boldsymbol{I}_{(n_c-2)} - \boldsymbol{L}_1 \otimes
\boldsymbol{I}_{(n_c-2)} \otimes \boldsymbol{I}_{(n_c-2)} -
\lambda^\tau \mtx{I}_{n_i},\]
where $\mtx{I}_{n_i}$ is the inverse of size $n_i \times n_i$,
$\boldsymbol{I}_{(n_c-2)}$ is the ${(n_c-2)} \times {(n_c-2)}$
identity matrix, and
\[\lambda^\tau = \frac{\max(\text{diag}(\mtx{C}_{ii})) +
\min(\text{diag}(\mtx{C}_{ii}))}2.\]
Unlike the original operator, the inverse of this homogenized operator
can be evaluated for little cost.  In fact, the inverse can be written
down explicitly in terms of the eigenvalues and eigenvectors of the
one-dimensional derivative matrix $\mtx{L}_1$.  So it can be
constructed for $O(n_c^3)$ cost.  This is in contrast to the at best
$O(n_c^6)$ cost of constructing the inverse of $\mtx{A}_{ii}^{\tau}$.

To explain the efficient inversion of the homogenized operator, let
$\mtx{L}_1 = \mtx{V}\mtx{E}\mtx{V}^{-1}$ denote the eigenvalue
decomposition of the one dimensional derivative matrix $\mtx{L}_1$;
i.e.  $\mtx{V}$ denotes the matrix whose columns are the eigenvectors
of $\mtx{L}_1$ and $\mtx{E}$ denotes the diagonal matrix where the
non-zero entries are the corresponding eigenvalues.  Then the
three-dimensional discrete Laplacian can be written as
\begin{equation*}
  \boldsymbol{L}_3 = \widetilde{\boldsymbol{V}} \boldsymbol{E}_\Delta
  \widetilde{\boldsymbol{V}}^{-1}
\end{equation*}
where 
\begin{align*}
\widetilde{\boldsymbol{V}} &:= \boldsymbol{V} \otimes
\boldsymbol{V} \otimes \boldsymbol{V},\\
\boldsymbol{E}_\Delta &:=
\boldsymbol{E} \otimes \boldsymbol{I}_{(n_c-2)} \otimes
\boldsymbol{I}_{(n_c-2)} + \boldsymbol{I} \otimes \boldsymbol{E} \otimes
\boldsymbol{I}_{(n_c-2)} \otimes + \boldsymbol{I}_{(n_c-2)} \otimes
\boldsymbol{I}_{(n_c-2)} \otimes \boldsymbol{E}\ {\rm and } \\
\widetilde{\boldsymbol{V}}^{-1} &:= \boldsymbol{V}^{-1} \otimes
\boldsymbol{V}^{-1} \otimes \boldsymbol{V}^{-1}.
\end{align*}
 This means that
\[\mtx{\widetilde{A}}^\tau_{ii} = \widetilde{\mtx{V}}
\left(-\mtx{E}_\Delta - \lambda\mtx{I}_{n_i}\right)
\widetilde{\mtx{V}}^{-1}.\]
Additionally, the inverse is given by
\[\mtx{\widetilde{A}}^{\tau,-1}_{ii} = \widetilde{\mtx{V}}
\left(-\mtx{E}_\Delta -
\lambda\mtx{I}_{n_i}\right)^{-1}\widetilde{\mtx{V}}^{-1}.\]
The inverse of $\left(-\mtx{E}_\Delta-\lambda\mtx{I}_{n_i}\right)$ can
be evaluated explicitly for little cost because $\mtx{E}_\Delta$
consists of the Kronecker product of diagonal matrices. Algorithm 3
outlines how $\mtx{\widetilde{A}}^{\tau,-1}_{ii}$ can be applied
rapidly to a vector.

\begin{remark}
  The blocks of the block-diagonal preconditioner presented in this
  section are built for each leaf independently.  Time in constructing the 
  preconditioner can be reduced by re-using the same local approximate 
  solver for leaves that have roughly the same medium.  However, the
  effectiveness will depend on how much the function $b(\vct{x})$ varies
  over the boxes where the homogenized operators are being reused.
\end{remark}

\begin{FloatingAlgorithm}
  \noindent\fbox{
\begin{minipage}{\textwidth}
\begin{center}
\textsc{Algorithm 3} $\left(\text{Fast application of }
\widetilde{\mtx{A}}^{\tau,-1}_{ii}\right)$
\end{center}
Let $\boldsymbol{v}$ be a vector of size $n_i=n_1^3$ (where $n_1 =
n_c-2$), the algorithm calculates $\boldsymbol{w} =
\widetilde{\boldsymbol{A}}^{\tau,-1}_{ii}
\boldsymbol{v}$.
\\
The function reshape$(\cdot,(p,q))$ is the classical reshape by
columns function, as found in \texttt{MATLAB} and \texttt{FORTRAN}
implementations.\\
\rule{\textwidth}{0.5pt}
\begin{tabbing}
\mbox{}\hspace{7mm} \= \mbox{}\hspace{6mm} \= \mbox{}\hspace{6mm} \=
\mbox{}\hspace{6mm} \= \mbox{}\hspace{6mm} \= \kill
(1) \> $\boldsymbol{M}_1 =
\text{reshape}\left(\boldsymbol{v},\left(n_1^2,n_1\right)\right)
\widetilde{\boldsymbol{V}}^{-1}$ \\
(2) \> \textbf{for} $j = 1,\dots,n_1$ \\
(3) \> \> $\boldsymbol{M}_2(:,j) =
\text{vec}\left(\widetilde{\boldsymbol{V}}^{-1}
  \text{reshape}\left(\boldsymbol{M}_1(:,j),\left(n_1,n_1\right)\right)
  \left(\widetilde{\boldsymbol{V}}^{-1}\right)^\intercal \right)$ \\
(4) \> \textbf{end for} \\
(5) \> $\boldsymbol{M}_1 =
\text{reshape}\left(\left(\boldsymbol{E}_\Delta - \lambda
\boldsymbol{I}_{i} \right)^{-1}
\text{vec}\left(\boldsymbol{M}_2\right),\left(n_1^2,n_1\right)\right)
\widetilde{\boldsymbol{V}}$ \\
(6) \> \textbf{for} $j = 1,\dots,n_1$ \\
(7) \> \> $\boldsymbol{M}_2(:,j) =
\text{vec}\left(\widetilde{\boldsymbol{V}}
\text{reshape}\left(\boldsymbol{M}_1(:,j),\left(n_1,n_1\right)\right)
\left(\widetilde{\boldsymbol{V}}\right)^\intercal \right)$ \\
(8) \> \textbf{end for} \\
(9) \> $\boldsymbol{w} = \text{vec}\left(\boldsymbol{M}_2\right)$
\end{tabbing}
\end{minipage}}
\end{FloatingAlgorithm}

\begin{FloatingAlgorithm}
\noindent\fbox{
\begin{minipage}{\textwidth}
\begin{center}
\textsc{Algorithm 4} $\left(\text{Application of the block-Jacobi
preconditioner } \boldsymbol{J}^{-1}\right)$
\end{center}
Let \[\boldsymbol{v} = \left(\begin{smallmatrix}
\boldsymbol{v}^{\tau=1} \\ \dots \\
\boldsymbol{v}^{\tau=N^\tau}\end{smallmatrix} \right)\] be a vector of
size $N^\tau n^\tau$, where \[\boldsymbol{v}^\tau =
\left(\begin{smallmatrix} \boldsymbol{v}^\tau_i \\
\boldsymbol{v}^\tau_b \end{smallmatrix}\right)\] is the sub-vector
associated to leaf $\tau$ (respectively $\boldsymbol{w},
\boldsymbol{w}^\tau$). The algorithm calculates $\boldsymbol{w} =
\boldsymbol{J}^{-1} \boldsymbol{v}$.\\
\rule{\textwidth}{0.5pt}
\begin{tabbing}
\mbox{}\hspace{7mm} \= \mbox{}\hspace{6mm} \= \mbox{}\hspace{6mm} \=
\mbox{}\hspace{6mm} \= \mbox{}\hspace{6mm} \= \kill
(1) \> \textbf{for} $\tau = 1, \dots, N_\ell$ \\
(2) \> \> calculate $\boldsymbol{f}_i =
\widetilde{\boldsymbol{A}}^{\tau,-1}_{ii} \boldsymbol{v}_i$ using
section \ref{sec:hom}, \\
(3) \> \> calculate $\boldsymbol{h}_b =
\boldsymbol{A}^\tau_{bi} \boldsymbol{f}_i -
\boldsymbol{v}_b$ using a sparse $\boldsymbol{A}^\tau_{bi}$,\\
(4) \> \> $\boldsymbol{g}_b = \widetilde{\mtx{S}}^{\tau,-1}
\boldsymbol{h}_b$ where $\widetilde{\mtx{S}}^{\tau,-1}$ has been
precomputed (see section \ref{sec:implementation}), \\
(5) \> \> $\boldsymbol{w}^\tau = \left(\begin{smallmatrix}
\boldsymbol{f}_i + \widetilde{\mtx{A}}^{\tau,-1}_{ii}
\boldsymbol{A}^\tau_{ib} \boldsymbol{g}_b \\
\boldsymbol{g}_b \end{smallmatrix}\right)$ using section \ref{sec:hom} and
sparse $\boldsymbol{A}^\tau_{ib}$.\\
(6) \> \textbf{end for}
\end{tabbing}
\end{minipage}}
\end{FloatingAlgorithm}

\subsection{Homogenization a preconditioner for the leaf
solve}\label{sec:homlocal}
This section illustrates the performance of the homogenization as a
preconditioner for a leaf box.  Since we are using a \texttt{GMRES}
solver, the clustering of the spectrum away from the origin is the key
to designing an effective preconditioner.

For illustration purposes, consider a leaf box of dimension
$0.25\times0.25\times0.25$ centered at $(0.125,0.125,0.125)$, where
the deviation from constant coefficient is $b(x) = -1.5
e^{-160(x^2+y^2+z^2)}$, $\kappa=40$ and the impedance boundary data is
given random complex values. Let
\begin{equation}\label{eqn:localsys}
  \mtx{A}_{\rm loc} \vct{w}= \vct{v}
\end{equation}
denote the discretized local problem corresponding to (\ref{eq:leaf}).

\begin{figure}
  \begin{subfigure}{0.49\textwidth}
    \includegraphics[width=\textwidth]{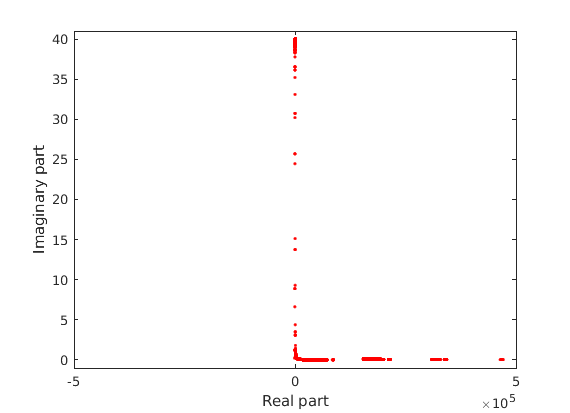}
    \caption{Unpreconditioned}
  \end{subfigure}
  \begin{subfigure}{0.49\textwidth}
    \includegraphics[width=\textwidth]{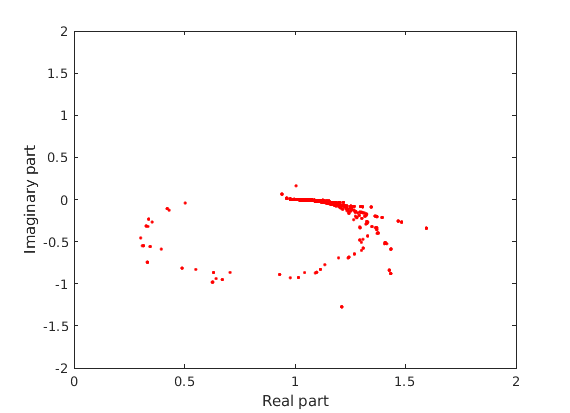}
    \caption{Preconditioned with an homogenized coefficient}
  \end{subfigure}
  \caption{Spectrum of a leaf matrix of dimension
    $0.25\times0.25\times0.25$ centered at $(0.125,0.125,0.125)$ with
    $b(x) = -1.5 e^{-160(x^2+y^2+z^2)}$, $\kappa=40$. 
  }
  \label{fig:localSU}
\end{figure}
Figure \ref{fig:localSU} reports on the spectrum of $\mtx{A}_{\rm
loc}$ without the preconditioner and the corresponding left
preconditioned problem.  It is clear that the preconditioner does an
excellent job of clustering the spectrum at $(1,0)$. Table
\ref{tab:localSU} reports the time for solving \eqref{eqn:localsys}
via backslash in \texttt{MATLAB} and using the preconditioned
iterative solution technique.  The iterative solution technique is 19
times faster than using backslash to solve the local problem.  These
results illustrate the preconditioner does an excellent job efficiently
handling local phenomena.

\begin{table}
  \centering
  \begin{tabular}{c|c|c|c}
    $n_c$ & Time backslash [s] & Time iterative [s] & Speed Up \\ \hline
    8     & 0.0108             & 0.0145             & 0.742 \\
    12    & 0.169              & 0.0258             & 6.53  \\
    16    & 0.815              & 0.0417             & 19.6  \\
  \end{tabular}
  \caption{Comparison of times between a direct inversion of
    $\mtx{A}_{\rm loc}$ and a preconditioned \texttt{GMRES}}
  \label{tab:localSU}
\end{table}

\section{Implementation}
\label{sec:par}
The code for the numerical experiments in section \ref{sec:numerics}
is written in \texttt{FORTRAN 90}.  This section describes the other
details about the implementation of the solution technique including
hardware, compilers, libraries and related software.

\subsection{Algebraic operations and sparse matrices} \label{sec:implementation}
The Intel \texttt{MKL} library with \texttt{complex} types are
utilized for the inversion of matrices, eigenvalue calculations,
sparse and full matrix-matrix and matrix-vector operations. The codes
are compiled using Intel \texttt{FORTRAN} compilers and libraries
version 20.2. We use the \texttt{MKL} \texttt{CSR} format and the
associated routines for algebraic operations with sparse matrices. The
matrices for the leaf boxes are never constructed explicitly. The
largest matrices that are constructed are the homogenized Schur
complement preconditioners; $\widetilde{\mtx{S}}^{\tau,-1}, \forall
\tau$; which are relatively small in size ($n_b \times n_b$).  These
homogenized Schur complement matrices are precomputed before the
iterative solver is applied.

\subsection{Distributed memory}
The parallelization and global solvers are accessed through the
\texttt{PETSC} library version 3.8.0. We use the \texttt{KSP}
implementation of \texttt{GMRES} present in \texttt{PETSC} and the
\texttt{MatCreateShell} interface for the matrix-free algorithms
representing $\mtx{A}$ and $\mtx{J}^{-1}$.  The message-passing
implementation between computing nodes uses Intel \texttt{MPI} version
$2017.0.098$.

\subsection{Hardware}
All experiments in Section \ref{sec:numerics}, were run on the
\texttt{RMACC Summit} supercomputer. The system has peak performance
of over 400 TFLOPS. The 472 general compute nodes each have 24 cores
aboard Intel Haswell CPUs, 128 GB of RAM and a local SSD.  This means
that each core has a limit of 4.6GB memory footprint.  

All nodes are connected through a high-performance network based on
Intel Omni-Path with a bandwidth of 100 GB/s and a latency of 0.4
microseconds. A 1.2 PB high-performance IBM GPFS file system is
provided.

\section{Numerical results}\label{sec:numerics}
This section illustrates the performance of the solution technique
presented in this paper.  In all the experiments, the domain $\Omega$
is the unit cube $[0,1]^3$.  While the variation in the medium can be
any smooth function, in this paper we consider the variation in the
medium $b(\vct{x})$ in \eqref{eqn:diffhelm} defined as
\begin{equation}
  b(\boldsymbol{x}) = -1.5 e^{-160 \left((x-0.5)^2+(y-0.5)^2+(z-0.5)^2\right)},
  \label{eqn:source}
\end{equation}
representing a Gaussian \emph{bump} centered in the middle of the unit
cube.  The performance of the solution technique will be similar for
other choices of $b(\vct{x})$.

The Chebyshev polynomial degree is fixed with $n_c = 16$. This follows
from the two dimensional discretization versus accuracy study in
\cite[Section 4.1]{BeamsGillmanHewett2020}. Also, this choice of
discretization order falls into the regime where the local
preconditioner was shown to be effective in section
\ref{sec:homlocal}.  For the scaling experiments, the solution is not
known.  For the experiments exploring the accuracy of the solution
technique the solution is known.  Throughout this section, we make
reference to several different terms.  For simplicity of presentation
they are defined here:

\noindent
$\bullet$ \textbf{ppw} denotes points per wavelength.  The ppw are
measured by counting the points along the axes and not across the
diagonal of the geometry.  The number of wavelengths across the
diagonal of the geometry is $\sqrt{3}$ times ppw.  \\
\noindent
$\bullet$ \textbf{its.} denotes the number of iterations that GMRES
took to achieve a specific result.\\
\noindent
$\bullet$ \textbf{Preconditioned residual reduction} (PCRR) denotes
the residual reduction stopping criteria set for \texttt{GMRES}. \\
\noindent
$\bullet$ \textbf{$r_k$} denotes the \textbf{preconditioned residual}
at iteration $k$ and is defined as
$$r_k = \|\mtx{J}^{-1} \left(\vct{b} - \mtx{A}
\vct{x}_k \right)\|_2$$ where $\vct{x}_k$ is the approximate solution
at iteration $k$.\\
\noindent
$\bullet$ The \textbf{Reference Preconditioned Residual Reduction
(RPRR)} denotes the stopping criteria for \texttt{GMRES} so that all
the possible digits attainable by the discretization are realized.\\
\noindent
$\bullet$ \textbf{$E_h$} denotes the \textbf{relative error} defined
by
\begin{equation*}
  E_h = \frac{\sqrt{\displaystyle\sum_{i=1}^{N }
      \left(u_h(\boldsymbol{x}_i) -
        u(\boldsymbol{x}_i)\right)^2}}{\sqrt{\displaystyle\sum_{i=1}^{N}
      u(\boldsymbol{x}_i)^2}},
\end{equation*}
where $N = N_\ell n$ is the total number of unknowns,
$u_h(\boldsymbol{x})$ is the approximate solution at the point
$\boldsymbol{x}$ obtained by allowing \texttt{GMRES} to \emph{run
until the accuracy no longer improves}, $u(\boldsymbol{x})$ is the
evaluation of the exact solution at the point $\boldsymbol{x}$, and
$N_\ell$ is the number of leaf boxes.  So the $E_h$ will stall at the
accuracy of the discretization. \\
\noindent
$\bullet$ \textbf{$E_h^{\rm it}$} denotes the relative error obtained
by the solution created by the iterative solver when with a fixed
residual reduction; i.e. stopping criterion.  As the residual
reduction decreases, $E_h^\text{it}$ converges to $E_h$.\\
\noindent
$\bullet$ \textbf{MPI Processes} (MPI Procs) are the computer
processes that run in parallel, exchanging information between them
using the Message Passing Interface (MPI). In our paper, each core
used contains only 1 MPI process and has access to a max of 4.6GB of
RAM.

This section begins by illustrating the scaling of the solver.  Next,
section \ref{sec:sep} explores the relationship between the accuracy
of the solution technique and the \texttt{GMRES} stopping criterion in
an effort to prevent from an excessive number of iterations being
performed.  A rule called the RPRR is developed for the stopping
criterion and tested on a different problem in section
\ref{sec:bumps}.

\subsection{Scaling}
\label{sec:scaling}
This section investigates the scaling of the parallel implementation
of the solver.  Both strong and weak scaling data is collected.
Additionally, we collected data to asses the scaling of the solver for
problems where the problem size in wavelengths per process is kept
fixed.  This last example closely aligns with how practitioners would
like to solve Helmholtz problems.

For all experiments in this section, the right hand side of
(\ref{eqn:diffhelm}) is $s(\boldsymbol{x})=b(\boldsymbol{x}) e^{i
\kappa \boldsymbol{x}}$. Figure \ref{fig:scaling} illustrates the
solution plotted on the planes $y=0.5$ and $z = 0.5$ for $\kappa=40,\
80,$ and $160$ with $4^3, \ 8^3$, and $16^3$ leaves, respectively.
\begin{figure}
  \centering
  \begin{tabular}{ccc}
  \includegraphics[width=0.32\textwidth]{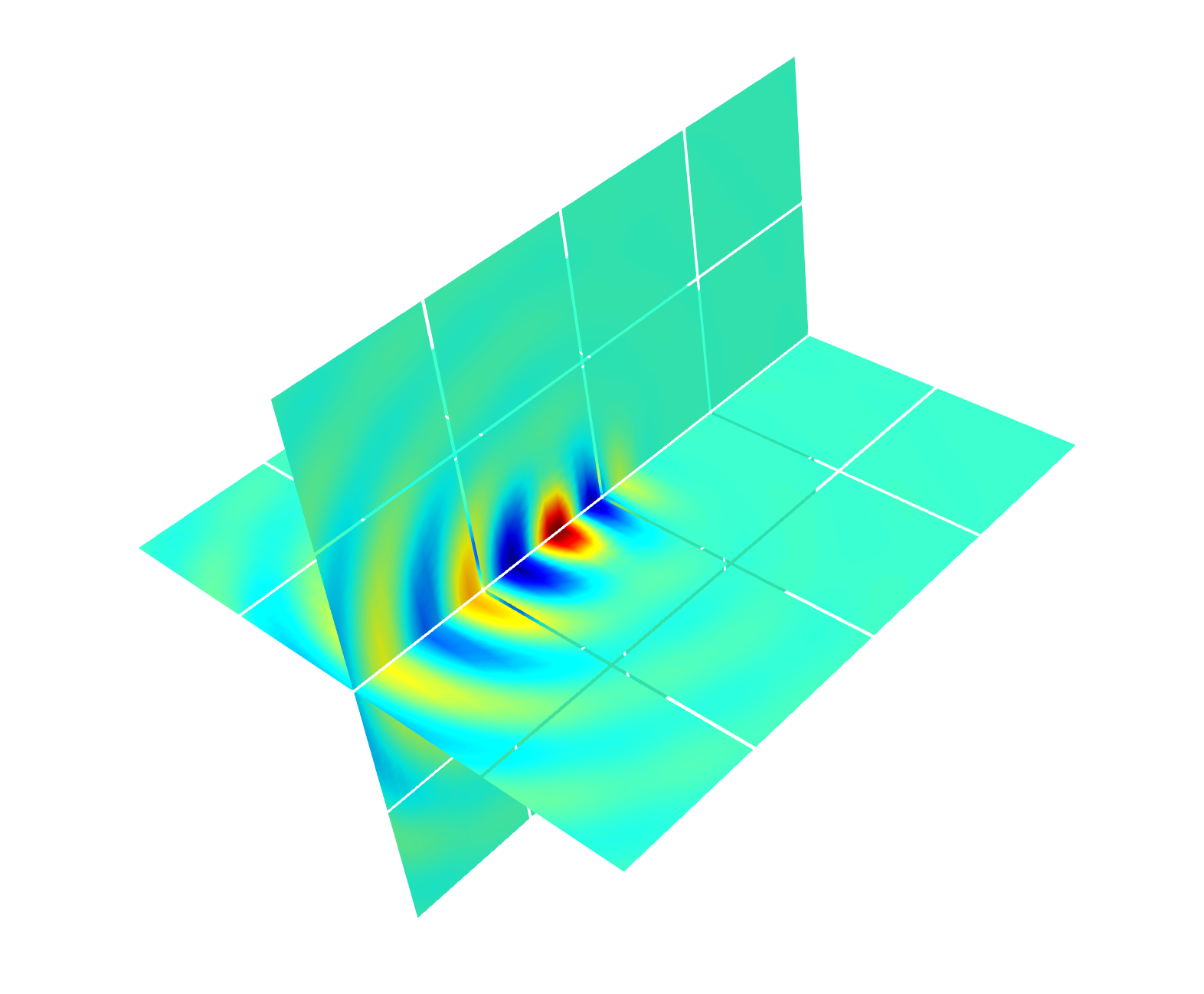} &
  \includegraphics[width=0.32\textwidth]{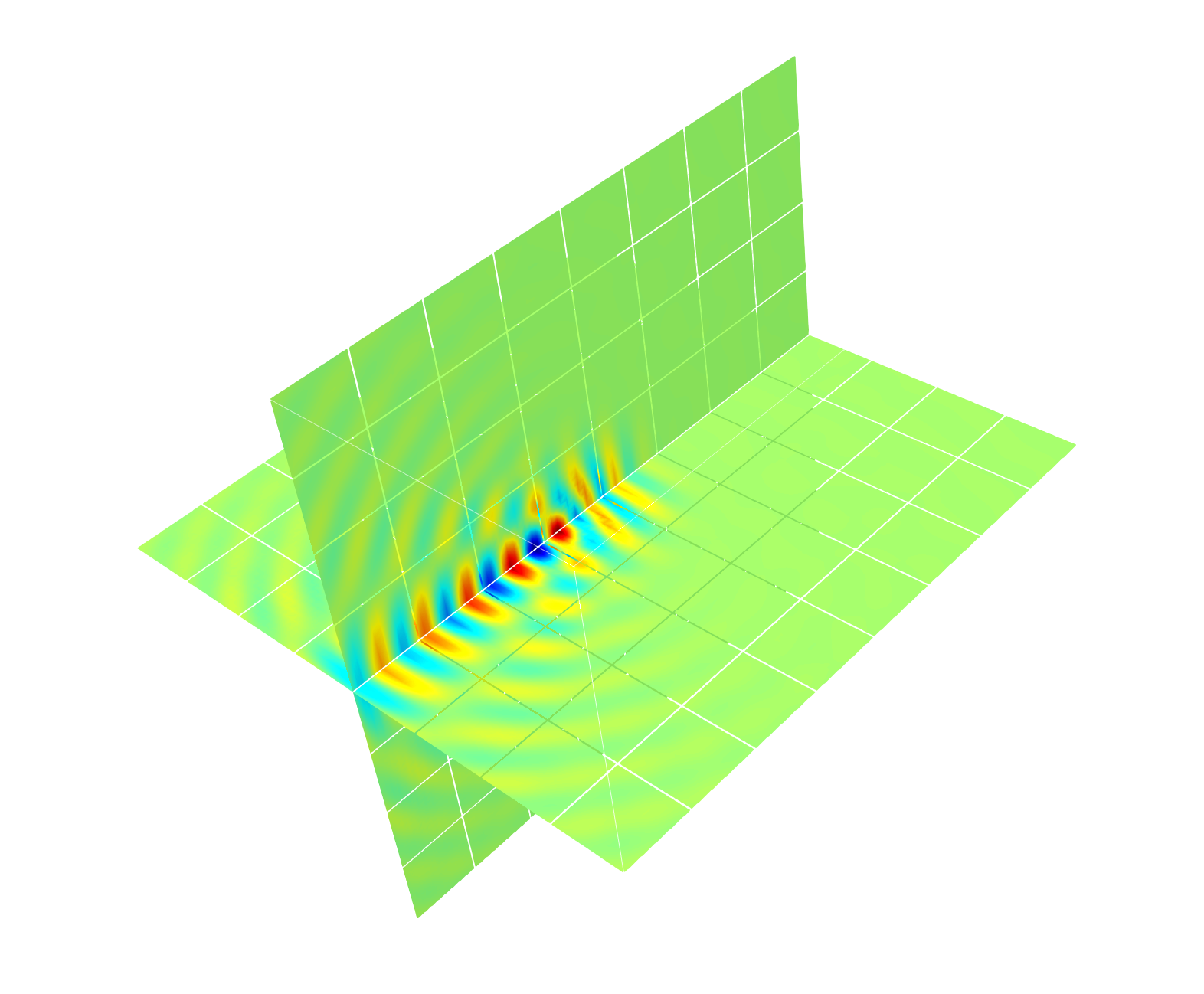}&
  \includegraphics[width=0.32\textwidth]{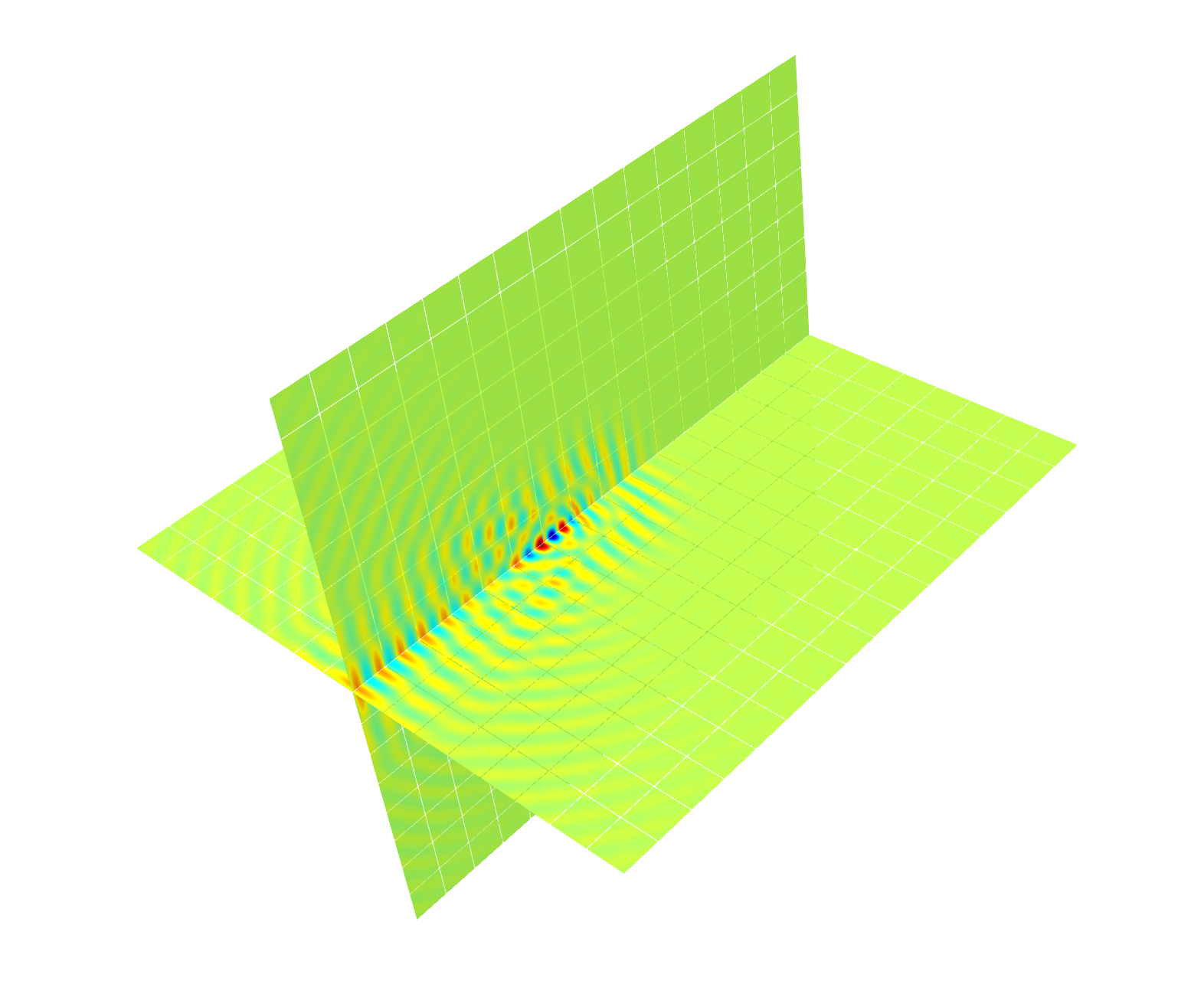}\\
  (a)& (b)&(c)
  \end{tabular}
  \caption{Illustration of the exact solution on the planes $y=0.5$
    and $z = 0.5$ for a problem with (a) $\kappa=40$ and $4^3$ leaves, (b)
    $\kappa=80$ and $8^3$ leaves, and (c) $\kappa=160$ and $16^3$ leaves
    from left to right. The $x-$axis runs diagonally to the right, the
    $y-$axis diagonally to the left and the $z-$axis is vertical.}
  \label{fig:scaling}
\end{figure}

\subsubsection{Strong scaling}
\label{sec:strong}
To investigate the strong scaling of the algorithm, we look at the
solution time versus the number of processes for a fixed problem size.
For this experiment, the stopping criterion for \texttt{GMRES} is
fixed at $10^{-8}$.  The problem under consideration is fixed with
$\kappa = 40$ with $8^3$ leaf boxes.  This corresponds to the solution
in Figure \ref{fig:scaling}(a).  The number of \texttt{MPI} processes
gets doubled in each experiment. Figure \ref{fig:strongscaling}
reports the scaling results.  The run times reduce from approximately
7 minutes for a sequential execution to well under a minute with the
minimum time happening with 16 \texttt{MPI} processes.  The 32
\texttt{MPI} processes case only has 2 leaves per process, which
implies that there is a lot of inter-process communication which
explains the increase in time between 16 and 32 \texttt{MPI}
processes.  Even so, the solution time remains under a minute.  This
is thanks to the fact the purely local operations dominate the
computations in this solution technique.  For the other experiments,
the run time approximately halves when the amount of processes is
doubled which is expected from a well implemented distributed memory
model.
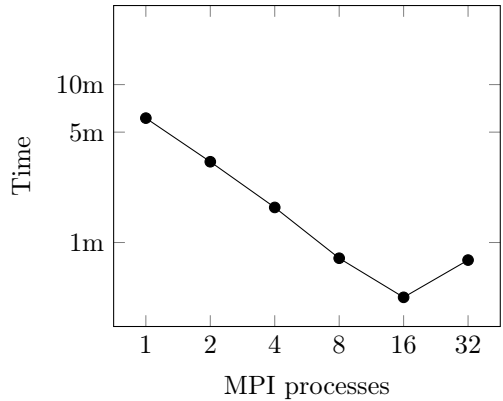
\begin{figure}
  \centering
  \begin{tikzpicture}
    \begin{loglogaxis}[
      scale=.75,
      xlabel={MPI processes},
      ylabel={Time},
      ytick={60,300,600},
      yticklabels={1m,5m,10m},
      ymin=0,
      ymax=1900,
      xtick={1,2,4,8,16,32},
      xticklabels={1,2,4,8,16,32}]
      \addplot[color=black,mark=*] coordinates {
        (32,0.46444190E+02)
        (16,0.26980971E+02)
        ( 8,0.47720490E+02)
        ( 4,0.10006856E+03)
        ( 2,0.19480623E+03)
        ( 1,0.36805203E+03)
      };
    \end{loglogaxis}
  \end{tikzpicture}
  \caption{A plot of the time in minutes versus the number of
    \texttt{MPI} processes when applying the presented solution technique
    to a problem discretized with $8^3$ leaf boxes to illustrate the strong scaling of the solution technique.}
  \label{fig:strongscaling}
\end{figure}

\subsubsection{Weak scaling}
\label{sec:weak}
Weak scaling investigates how the solution time varies with the number
of processes for a fixed problem size per process and a fixed residual
reduction PCRR.  The discretization is refined and at the same time
more \texttt{MPI} processes are used in order to keep the number of
degrees of freedom per process constant.  Two experiments are
considered for weak scaling: a fixed wave number for all experiments
and an increased wave number to maintain a fixed number of points per
wavelength (ppw) while the mesh is refined.  The latter is
representative of the type of experiments desired by practitioners.
Figure \ref{fig:weakscaling} reports the weak scaling for when the
local problem size is fixed to $4^3$; i.e. $4^3$ leaf boxes are
assigned to each process.
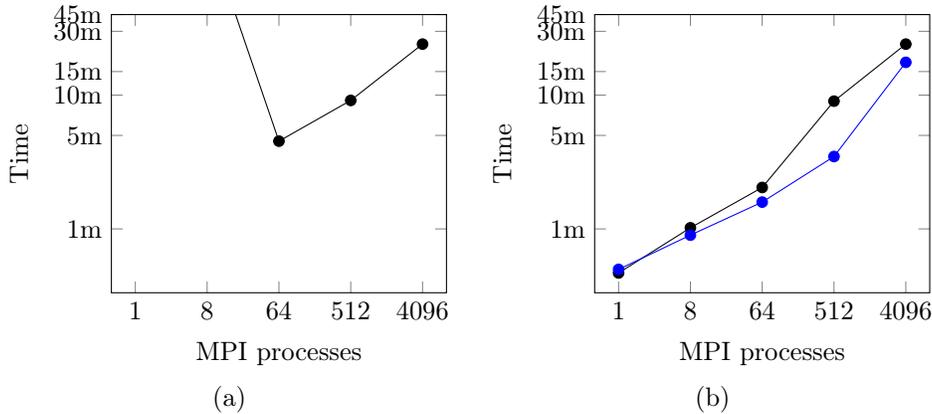
\begin{figure}
\begin{tabular}{cc}
  \begin{tikzpicture}
    \begin{loglogaxis}[
      scale=0.65,
      xlabel={MPI processes},
      ylabel={Time},
      ytick={60,300,600,900,1800,2400},
      yticklabels={1m,5m,10m,15m,30m,45m},
      ymin=20,
      ymax=2400,
      xtick={1,8,64,512,4096},
      xticklabels={1,8,64,512,4096},
      xmin=0.5,
      xmax=8192]
      \addplot[color=black,mark=*] coordinates {
        (   8,1E+04)
        (  64,0.27172773E+03)
        ( 512,0.54778151E+03)
        (4096,0.14438151E+04)
      };
    \end{loglogaxis}
  \end{tikzpicture}
 &   \begin{tikzpicture}
    \begin{loglogaxis}[
      scale=0.65,
      xlabel={MPI processes},
      ylabel={Time},
      ytick={60,300,600,900,1800,2400},
      yticklabels={1m,5m,10m,15m,30m,45m},
      ymin=20,
      ymax=2400,
      xtick={1,8,64,512,4096},
      xticklabels={1,8,64,512,4096},
      xmin=0.5,
      xmax=8192]
      \addplot[color=black,mark=*] coordinates {
        (   1,0.28114729E+02)
        (   8,0.61147678E+02)
        (  64,0.12244576E+03)
        ( 512,0.54093485E+03)
        (4096,0.14438151E+04)
      };
      \addplot[color=blue,mark=*] coordinates {
        (   1,0.29887427E+02)
        (   8,0.53907974E+02)
        (  64,0.95190765E+02)
        ( 512,0.20858756E+03)
        (4096,0.10550612E+04)
      };
    \end{loglogaxis}
  \end{tikzpicture}\\
  (a) & (b)
\end{tabular}
\caption{Illustration of the weak scaling of the algorithm.  The time in minutes for solving the Helmholtz problem 
  (a) with a fixed wave number problem with $\kappa = 320$ and 
  (b) maintaining $16$ ppw (black) and $8$ ppw (blue) on a geometry with a fixed $4\times 4\times
  4$ leaves local problem on each process.
}
\label{fig:weakscaling}
\end{figure}

Figure \ref{fig:weakscaling}(a) reports on the performance when the
wave number fixed to $\kappa = 320$ and the residual reduction is set
to $10^{-10}$. The problems range from 4 to 1 a wavelength per
leaf. The experiments with a small number of processes correspond to
many wavelengths per leaf. For example, 64 MPI processes corresponds
to 4 wavelengths per leaf while 8 MPI processes corresponds to 8
wavelengths per leaf.  It is well-known that local homogenization will
not work well \cite{Chaumont2015} for the experiments with a high
number of wavelengths per leaf.  As the number of wavelengths per leaf
decreases, the homogenized preconditioner performs well and the
performance of the global solver is significantly better. It does
reach a point where the processes are saturated and communication
causes an increase in the timings.

Figure \ref{fig:weakscaling}(b) reports on the weak scaling
performance for experiments where the wave number is increased to
maintain $16$ and $8$ ppw plotted in black and blue, respectively.
For the $16$ ppw example, the residual reduction is fixed at
$10^{-10}$ for all cases resulting is a solution that is accurate to a
maximum of $8$ digits (see section \ref{sec:sep}).  For the $8$ ppw
example, the residual reduction is fixed at $10^{-8}$ for all cases
resulting in a solution that is accurate to approximately $4$ digits.
A flat curve is not expected in this weak scaling experiment.  To
understand this consider the Fourier \emph{modes} of the
\emph{residual}.  The block-Jacobi preconditioner is designed to
tackle local phenomena on each leaf which corresponds to high
frequency information in the \emph{residual}.  The lower-frequency
modes in the \emph{residual} remain relatively unpreconditioned.
Tackling these lower-frequency modes in the residual would require a
\emph{global} filter, i.e. an appropriate coarse solver, which is an
active area of research for Helmholtz problems. See
\cite{Brandt1977,Brandt1994} for general local Fourier error analysis,
\cite{Hemker2004,Hemker2003,LuceroGander2020} for block-Jacobi
preconditioners on non-overlapping discretizations and \cite[section
4]{GanderZhang2019} and \cite{ErnstGander2012} on scalable Helmholtz
solvers and preconditioners.

The runtimes for the fixed number of ppw experiments grow roughly
linearly for larger problems. This is expected since we use an
iterative method where information is only exchanged between
neighboring leaves. Necessarily, the number of steps needed to achieve
a fixed residual reduction must be at least equal to the distance
between opposite sides of the domain, counted in leaves sharing a face
\cite[p.17]{ToselliWidlund2005}. Ultimately we expect the scaling to
be superlinear, but a linear behavior for regimes involving a billion
degrees of freedom is very promising.

The largest problems under consideration correspond to the $4096$ MPI
processes experiments in Figure \ref{fig:scaling}(b) and over a
billion unknowns.  For the $16$ ppw example, this is a problem that is
approximately $50\times 50\times 50$ wavelengths in size and it takes
the solver approximately $24$ minutes to achieve the set accuracy.
For the $8$ ppw example, this problem is $100\times 100\times 100$
wavelengths in size and it takes the solver approximately $17$ minutes
to achieve the set accuracy.  This means that by lowering the desired
accuracy of the approximation, the solution technique can solve a
problems twice the size in wavelengths with the same computational
resources in less time.

Tables \ref{tab:mem16ppw} and \ref{tab:mem8ppw} report the scaling
performance including the memory usage for the experiments in Figure
\ref{fig:scaling}(b).  In these tables, it is easy to see the rough
doubling of the timings and the number of iterations.  It is also
important to note that the increase in memory per process is low as
the problem grows.  This is what allowed us to solve large problems on
the \texttt{RMACC Summit} cluster nodes which allows only 4.6 GB per
process.  In fact, the memory per process is under 3GB for all the
experiments.  This low memory footprint is thanks in large part to the
exploitation of the Kronecker product structure and the sparsity of
operators.
\begin{table}
  \centering
  \resizebox{\columnwidth}{!}{%
  \begin{tabular}{r|r|r|r|r|r}
    Problem size in leaves & $4\times 4\times 4$ & $8\times 8\times 8$ & $16\times 16\times 16$ & $32\times 32\times 32$ & $64\times 64\times 64$ \\ \hline
    Problem size in wavelengths & $3$ & $6$ & $12$ & $24$ & $48$ \\ 
    MPI procs                         & 1       & 8       & 64      & 512     & 4096    \\
    Degrees of freedom $N$            & 250k    & 2M      & 16M     & 128M    & 1027M   \\
    GMRES iterations                  & 156     & 216     & 402     & 736     & 1478    \\
    GMRES time                        & 28s     & 61s     & 122s    & 541s    & 1444s   \\
    Peak memory per process           & 1.94 GB & 2.21 GB & 2.29GB  & 2.38 GB & 2.70 GB
  \end{tabular}}
  \caption{Problem size in wavelengths, \texttt{MPI} processes (MPI procs), number of
    discretization points, number of \texttt{GMRES} iterations, time for
    the iterative solver to converge with a preconditioned relative
    residual of $10^{-10}$, and the peak memory per process for different
    problem sizes.  Here the wave number is increased to maintain $16$
    ppw.}
  \label{tab:mem16ppw}
\end{table}

\begin{table}
  \centering
  \resizebox{\columnwidth}{!}{%
  \begin{tabular}{r|r|r|r|r|r}
    Problem size in leaves & $4\times 4\times 4$ & $8\times 8\times 8$ & $16\times 16\times 16$ & $32\times 32\times 32$ & $64\times 64\times 64$ \\ \hline
    Problem size in wavelengths & $6$ & $12$ & $24$ & $48$ & $96$ \\ 
    MPI procs                         & 1       & 8       & 64      & 512     & 4096    \\
    Degrees of freedom $N$            & 250k    & 2M      & 16M     & 128M    & 1027M   \\
    GMRES iterations                  & 153     & 191     & 315     & 567     & 1065    \\
    GMRES time                        & 30s     & 54s     & 95s     & 209s    & 1055s   \\
    Peak memory per process           & 1.92 GB & 2.21 GB & 2.29GB  & 2.37 GB & 2.70 GB
  \end{tabular}}
  \caption{Problem size in wavelengths, \texttt{MPI} processes (MPI procs), number of
    discretization points, number of \texttt{GMRES} iterations, time for
    the iterative solver to converge with a preconditioned relative
    residual of $10^{-8}$, and the peak memory per process for different
    problem sizes.  Here the wave number is increased to maintain $8$
    ppw.}
  \label{tab:mem8ppw}
\end{table}

\subsection{Relative residual rule for maximum accuracy}
\label{sec:sep}This section investigates the necessary stopping
criterion for \texttt{GMRES} in order to guarantee that all the digits
that are achievable by the discretization.  To do this, we consider
the boundary value problem with an exact solution of
\begin{equation}
  u(x,y,z) = e^{i \kappa (x + y + z)} e^{x} \cosh(y) (z + 1)^2
  \label{eqn:plane}
\end{equation}
for different wave numbers $\kappa$ and the variable coefficient
$b(\boldsymbol{x})$ as defined in equation \eqref{eqn:source}. Figure
\ref{fig:plane} illustrates the solution for different values of
$\kappa$.  The exact solution (\ref{eqn:plane}) is a plane wave in the
direction $(1,1,1)$ that is modulated differently along the $x$, $y$
and $z$ axes.

\begin{figure}
  \centering
  \begin{tabular}{ccc}
  \includegraphics[width=0.32\textwidth]{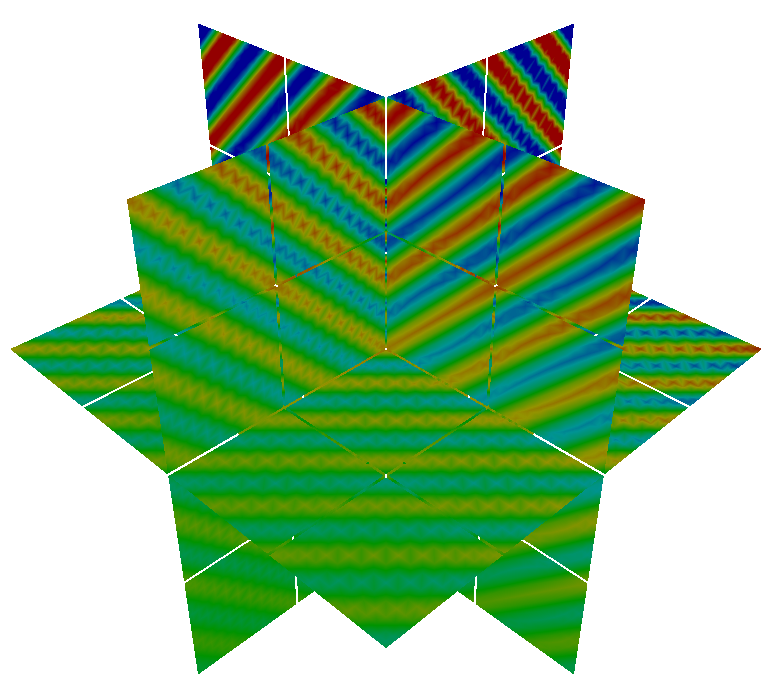}&
  \includegraphics[width=0.32\textwidth]{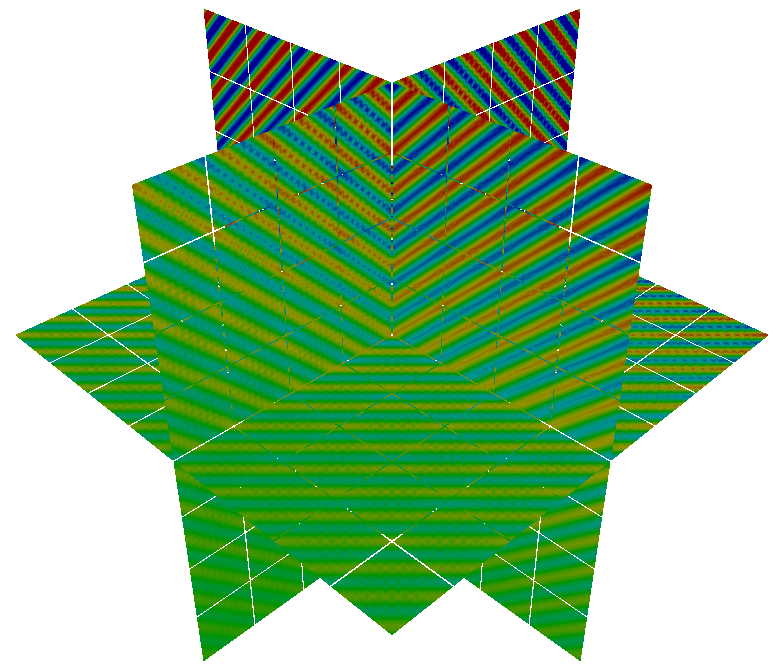}&
  \includegraphics[width=0.32\textwidth]{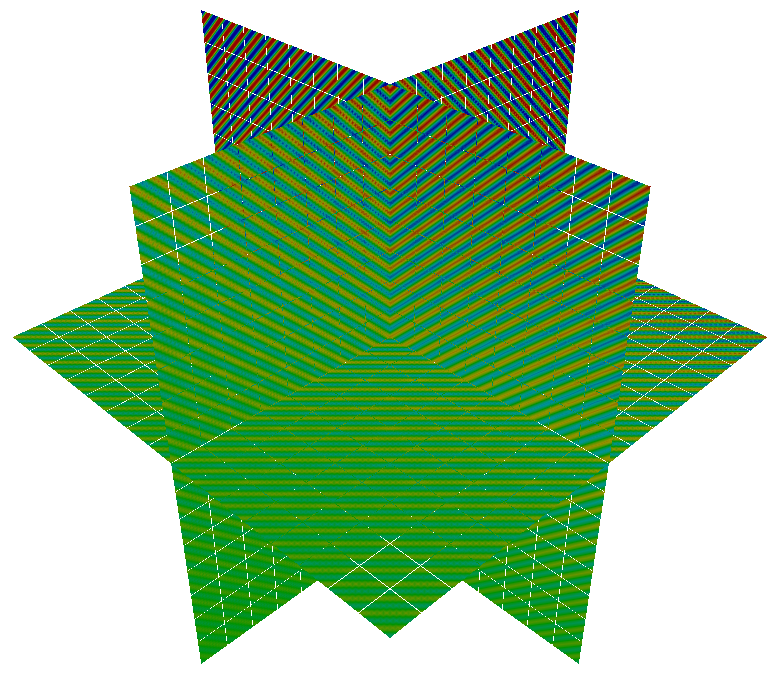}\\
  (a)&(b)&(c)
  \end{tabular}
  \caption{Illustration of the exact solution \eqref{eqn:plane} when
    $9.6$ ppw is maintained for $4^3$, $8^3$ and $16^3$ leaves (from left
    to right).  The plots of the solution are on the planes $x=0.5$,
    $y=0.5$ and $z = 0.5$ where the $x-$axis runs diagonally to the right,
    the $y-$axis diagonally to the left and the $z-$axis is vertical.}
  \label{fig:plane}
\end{figure}

Table \ref{tab:PlaneMinErr} reports the relative error $E_h$ and
digits of accuracy obtained by the discretization when the number of
ppw remains fixed across a row.  The error $E_h$ related to the ppw
remains fixed.  This is consistent with the accuracy observed when
applying this discretization to two dimensional problems. The values
range from orders of magnitude of $10^{-12}$ for 24 ppw to $10^{-7}$
for 9.6 ppw.

\begin{table}
  \centering
  \caption{The relative error $E_h$ and number of digits
    of accuracy obtained for different ppw and number of leaf boxes for the problem considered in section \ref{sec:sep} .}
  \begin{tabular}{c|c|c|c|c|c}
    \diagbox{ppw}{Leaves} & $4^3$ & $8^3$ & $16^3$ & $32^3$ & Digits\\\hline
    24.0 &\texttt{2.057E-12}&\texttt{1.763E-12}&\texttt{1.715E-12}&\texttt{1.727E-12}& 11\\
    19.2 &\texttt{8.872E-11}&\texttt{9.166E-11}&\texttt{9.275E-11}&\texttt{9.360E-11}& 10\\
    16.0 &\texttt{1.514E-09}&\texttt{1.922E-09}&\texttt{1.631E-09}&\texttt{1.667E-09}& 8\\
    12.0 &\texttt{7.121E-08}&\texttt{7.246E-08}&\texttt{7.221E-08}&\texttt{7.313E-08}& 7\\
    9.60 &\texttt{6.932E-07}&\texttt{5.631E-07}&\texttt{4.836E-07}&\texttt{4.478E-07}& 6\\
  \end{tabular}
  \label{tab:PlaneMinErr}
\end{table}

By looking at the residual reduction at each iteration of
\texttt{GMRES} we are able to define a stopping criterion which
prevents extra iterations from being performed but the accuracy of the
discretization is achieved.  As defined in the beginning of section
\ref{sec:numerics}, we call this the \textit{RPRR}.  Table
\ref{tab:PlaneResRed} reports the RPRR and the corresponding relative
iterative error $E_h^\text{it}$ when using this stopping criterion for
the same problems as in Table \ref{tab:PlaneMinErr}.  The relative
errors reported in both tables match in terms of digits.

Table \ref{tab:PlaneErrT} reports the time in seconds for the solver
and the number of \texttt{GMRES} iterations needed to obtain the
residual reduction in Table \ref{tab:PlaneResRed}.  The results from
this section suggest that it is sufficient to ask for the residual
reduction to be about two digits smaller than the expected accuracy of
the discretization.  Another observation from Table
\ref{tab:PlaneErrT} is that while the problem at 9.60 ppw has a
solution that is more oscillatory than the problem at 24 ppw the
solution technique converges faster.  This is because a lower accuracy
solution is obtained and the stopping criterion for \texttt{GMRES} is
also lower.  Looking left-to-right in this table, the number of
iterations and time to convergence increases.  This is due to the
global problem becoming larger in number of wavelengths as the number
of leaf boxes increases and the need to propagate information across
all the boxes.
\begin{table}
  \centering
  \caption{The RPRR needed to guarantee that the accuracy of the discretization is achieved
  in for the experiments in Table \ref{tab:PlaneMinErr}.  The relative error $E_h^\text{it}$ 
  is reported when using the RPRR as the stopping criteria for \texttt{GMRES}.
}
  \label{tab:PlaneResRed}
  \begin{tabular}{c|c|c|c|c|c}
    \diagbox{ppw}{Leaves} & $4^3$ & $8^3$ & $16^3$ & $32^3$ & RPRR \\\hline
    24.0 &\texttt{8.122E-13}&\texttt{1.632E-12}&\texttt{8.844E-13}&\texttt{5.231E-13}&$10^{-13}$\\
    19.2 &\texttt{8.883E-12}&\texttt{6.632E-12}&\texttt{3.534E-12}&\texttt{1.917E-12}&$10^{-12}$\\
    16.0 &\texttt{1.275E-09}&\texttt{1.681E-09}&\texttt{1.021E-09}&\texttt{5.151E-10}&$10^{-10}$\\
    12.0 &\texttt{1.438E-08}&\texttt{1.145E-08}&\texttt{7.215E-09}&\texttt{3.326E-09}&$10^{- 9}$\\
    9.60 &\texttt{2.541E-07}&\texttt{1.385E-07}&\texttt{8.280E-08}&\texttt{4.408E-08}&$10^{- 8}$\\
  \end{tabular}
\end{table}

\begin{table}
  \centering
  \caption{The time in seconds it takes \texttt{GMRES} and number of iterations
it takes \texttt{GMRES} to converge with the RPRR as the stopping criterion from table \ref{tab:PlaneResRed}
for various points per wavelength and meshes.}
  \label{tab:PlaneErrT}
  \begin{tabular}{r|r|r|r|r|r|r|r|r}
    \multirow{2}{*}{\diagbox{ppw}{Leaves}} & \multicolumn{2}{c|}{$4^3$} & \multicolumn{2}{c|}{$8^3$} & \multicolumn{2}{c|}{$16^3$} & \multicolumn{2}{c}{$32^3$} \\\cline{2-9}
         &time[s]& its.&time[s]& its.&time[s]& its.&time[s]& its.\\\hline
    24.0 &    48 & 216 &    97 & 296 &   351 & 477 &  947  & 775  \\
    19.2 &    55 & 187 &   125 & 263 &   278 & 392 &  869  & 668  \\
    16.0 &    35 & 146 &    96 & 184 &   389 & 277 &  482  & 466  \\
    12.0 &    40 & 141 &    78 & 165 &   281 & 243 &  554  & 407  \\
    9.6  &    38 & 140 &    38 & 142 &   235 & 191 &  447  & 320  \\
  \end{tabular}
\end{table}

\subsection{Effectiveness of the proposed rule for stopping criteria}
\label{sec:bumps} This section verifies the effectiveness of the RPRR
proposed in section \ref{sec:sep} as the stopping criteria for
\texttt{GMRES}.  That is, the desired residual reduction is set to 2-3
digits more than the accuracy that can be achieved by the
discretization.

The experiments in this section have an exact solution given by
\begin{equation}\label{eqn:bumps}
  u(x,y,z) = \left(1+e^{i \kappa x}\right)\left(1+e^{i \kappa
      y}\right)\left(1+e^{i \kappa z}\right) \ln\left(1 + x^2 + y^2 +
    z^2\right)
\end{equation}
which is not separable. The variable coefficient $b(\boldsymbol{x})$
is defined in equation \eqref{eqn:source}.  Figure \ref{fig:bumps}
illustrates the solution \eqref{eqn:bumps} for different values of
$\kappa$.  Roughly speaking, it is a grid of bumps along any axis that
is modulated by the function $\ln\left(1 + x^2 + y^2 +
z^2\right)$. The pattern of bumps makes it easy to observe the number
of wavelengths along the axes.  Given that this solution is more
oscillatory (with roughly the same magnitude) throughout the geometry
than the solution in the previous section, it is expected that it will
take more iterations for \texttt{GMRES} to converge.  This is
especially true for high frequency problems.
\begin{figure}
  \centering
  \begin{tabular}{ccc}
  \includegraphics[width=0.32\textwidth]{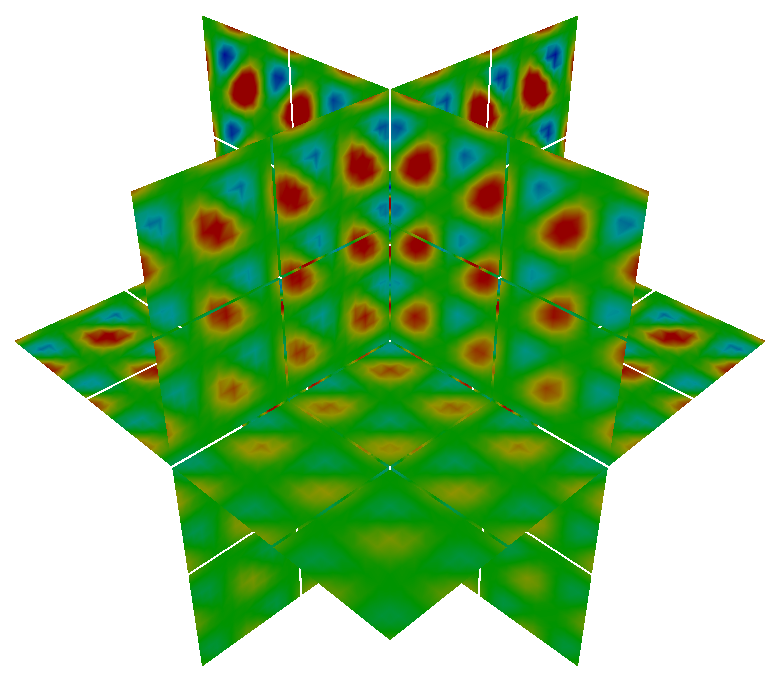} &
  \includegraphics[width=0.32\textwidth]{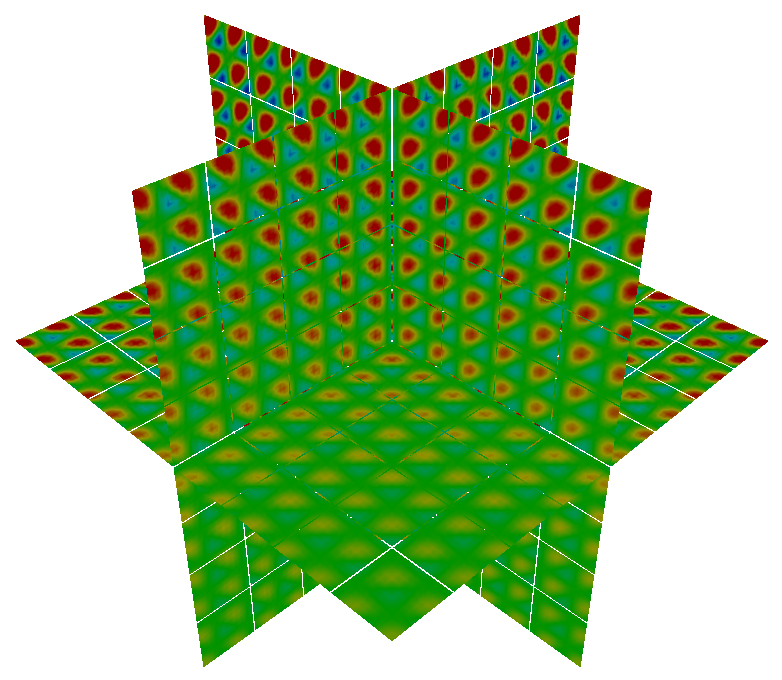}&
  \includegraphics[width=0.32\textwidth]{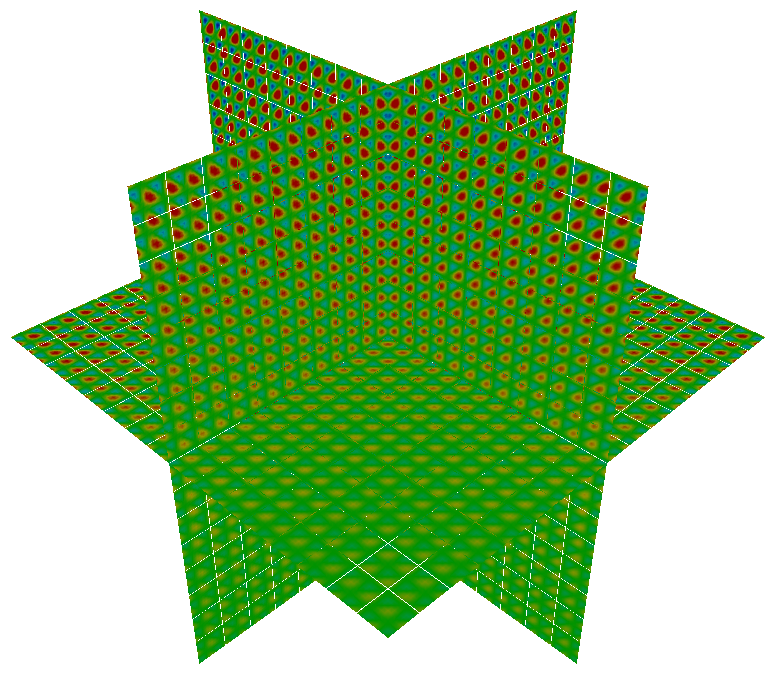}\\
  (a)&(b)&(c)
  \end{tabular}
  \caption{Illustration of the \textit{bumps} solution
    \eqref{eqn:bumps} cut at the planes $x=0.5$, $y=0.5$ and $z = 0.5$ for
    $9.6$ ppw with discretizations consisting of (a) $4^3$, (b) $8^3$ and
    (c) $16^3$ leaves.  The $x-$axis runs diagonally to the right, the
    $y-$axis diagonally to the left and the $z-$axis is vertical.}
  \label{fig:bumps}
\end{figure}

Table \ref{tab:BumpsMinErr} reports the relative error $E_h^\text{it}$
obtained by using the relative residual stopping criteria prescribed
in the Section \ref{sec:sep}.  The orders of magnitude in the error
are similar to what was observed for the separable solution in Table
\ref{tab:PlaneMinErr} confirming the accuracy of the discretization.
Table \ref{tab:BumpsMinErr2} reports the maximum achievable accuracy
of the discretization $E_h$.  The two errors are of the same order.
This means that the stopping criteria is giving the maximum possible
accuracy attainable by the discretization.
\begin{table}
  \centering
  \caption{The relative error $E_h^\text{it}$ obtained for
    different ppw and discretizations using the RPRR
    rule for the boundary value problem with solution given in (\ref{eqn:bumps}).}
  \begin{tabular}{c|c|c|c|c}
    \diagbox{ppw}{Leaves} & $4^3$ & $8^3$ & $16^3$ & $32^3$ \\\hline
    24.0 &\texttt{3.647E-12}&\texttt{3.957E-12}&\texttt{6.770E-12}&\texttt{1.270E-11} \\
    19.2 &\texttt{1.456E-10}&\texttt{2.933E-10}&\texttt{5.432E-10}&\texttt{1.004E-09} \\
    16.0 &\texttt{2.930E-09}&\texttt{7.184E-09}&\texttt{1.377E-08}&\texttt{2.869E-08} \\
    12.0 &\texttt{1.528E-07}&\texttt{2.629E-07}&\texttt{4.573E-07}&\texttt{7.196E-07} \\
    9.60 &\texttt{5.389E-06}&\texttt{1.614E-06}&\texttt{2.857E-06}&\texttt{4.461E-06} \\
  \end{tabular}
  \label{tab:BumpsMinErr}
\end{table}

\begin{table}
  \centering
  \caption{The relative error $E_h$ obtained for different ppw
    and discretizations for the boundary value problem with solution (\ref{eqn:bumps}).  Here 
    \texttt{GMRES} was allowed to run until it the residual reduction was machine precision.}
  \begin{tabular}{c|c|c|c|c}
    \diagbox{ppw}{Leaves} & $4^3$ & $8^3$ & $16^3$ & $32^3$ \\\hline
    24.0 &\texttt{3.738E-12}&\texttt{3.856E-12}&\texttt{6.736E-12}&\texttt{1.263E-11}\\
    19.2 &\texttt{1.480E-10}&\texttt{3.035E-10}&\texttt{5.432E-10}&\texttt{1.004E-09}\\
    16.0 &\texttt{2.910E-09}&\texttt{6.588E-09}&\texttt{1.376E-08}&\texttt{2.861E-08}\\
    12.0 &\texttt{1.528E-07}&\texttt{2.630E-07}&\texttt{4.573E-07}&\texttt{7.197E-07}\\
    9.60 &\texttt{1.302E-06}&\texttt{1.614E-06}&\texttt{2.856E-06}&\texttt{4.458E-06}\\
  \end{tabular}
  \label{tab:BumpsMinErr2}
\end{table}

Table \ref{tab:BumpsErrT} reports the times and number of iterations
needed for \texttt{GMRES} to converge with the prescribed residual
reduction.  As expected, since the solution to the problem in this
section is more oscillatory than in the previous section, more
iterations are needed to achieve the same accuracy with the same
discretizations. All other trends observed in the experiments from
section \ref{sec:sep} remain the same.

\begin{table}
  \centering
  \caption{The time in seconds and number of iterations needed for
    \texttt{GMRES} to achieve maximum accuracy using the RPRR rule for the
    boundary value problem with solution (\ref{eqn:bumps}.}
  \label{tab:BumpsErrT}
  \begin{tabular}{r|r|r|r|r|r|r|r|r}
    \multirow{2}{*}{\diagbox{ppw}{Leaves}} & \multicolumn{2}{c|}{$4^3$} & \multicolumn{2}{c|}{$8^3$} & \multicolumn{2}{c|}{$16^3$} & \multicolumn{2}{c}{$32^3$} \\\cline{2-9}
         &time[s]& its.&time[s]& its.&time[s]& its.&time[s]& its.\\\hline
    24.0 &    58 & 242 &   90  & 352 &  493  & 559 &  969  & 948 \\
    19.2 &    42 & 205 &  112  & 288 &  281  & 456 &  929  & 827 \\
    16.0 &    39 & 170 &   92  & 236 &  205  & 379 &  702  & 663 \\
    12.0 &    60 & 168 &   80  & 201 &  163  & 320 &  610  & 557 \\
    9.6  &    75 & 163 &  138  & 190 &  310  & 280 &  531  & 480 \\
  \end{tabular}
\end{table}

\section{Concluding remarks}
\label{sec:con}
This manuscript presented an efficient iterative solution technique
for the linear system that results from the HPS discretization of
variable coefficient Helmholtz problems.  The technique is based on a
block-Jacobi preconditioner coupled with \texttt{GMRES}.  The
preconditioner is based on local homogenization and is extremely
efficient to apply thanks to the tensor product nature of the
discretization.  The homogenization is highly effective because the
elements (or leaves) are never many wavelengths in size.

The numerical results illustrate the effectiveness of the solution
technique.  While the method does not ``scale'' in the parallel
computing context, it is extremely efficient.  For example, it is able
to solve a three dimensional mid frequency Helmholtz problem (roughly
100 wavelengths in each direction) requiring over 1 billion
discretization points to achieve approximately 4 digits of accuracy in
under 20 minutes on a cluster.  The numerical results also illustrate
that the HPS method is able to achieve a prescribed accuracy at set
ppw.  This is consistent with the accuracy observed for two
dimensional problem \cite{BeamsGillmanHewett2020,
GillmanBarnettMartinsson2015}. Additionally, the numerical results
indicate that it is sufficient to ask for the relative residual of
approximately 2-3 digits more than the accuracy expected of the
discretization. This will minimize the number of extra iterations
performed by \texttt{GMRES} but does not effect the accuracy of the
solution technique.

As mentioned in the text, the proposed solver does not scale.  In
future work, one could develop a multi-level solver where the
block-Jacobi preconditioner will be used as smoother between the
coarse and the fine grids.  Efficient solvers for the coarse grid is
ongoing work.

\section*{Acknowledgments}
The authors wish to thank Total Energies for their permission to
publish.  The work by A. Gillman is supported by the National Science
Foundation (DMS-2110886). A. Gillman, J.P. Lucero Lorca and N. Beams
are supported in part by a grant from Total Energies.

\bibliographystyle{amsplain}
\bibliography{main}

\providecommand{\bysame}{\leavevmode\hbox to3em{\hrulefill}\thinspace}
\providecommand{\MR}{\relax\ifhmode\unskip\space\fi MR }
\providecommand{\MRhref}[2]{%
  \href{http://www.ams.org/mathscinet-getitem?mr=#1}{#2}
}
\providecommand{\href}[2]{#2}
\begin{thebibliography}{10}

\bibitem{BabuskaIhlenburgPaikSauter1995}
Ivo Babuška, Frank Ihlenburg, Ellen~T. Paik, and Stefan~A. Sauter, \emph{A
  generalized finite element method for solving the {H}elmholtz equation in two
  dimensions with minimal pollution}, Computer Methods in Applied Mechanics and
  Engineering \textbf{128} (1995), no.~3, 325--359.

\bibitem{BabuskaIvoSauter1997}
Ivo~M. Babuška and Stefan~A. Sauter, \emph{Is the pollution effect of the fem
  avoidable for the {H}elmholtz equation considering high wave numbers?}, SIAM
  Journal on Numerical Analysis \textbf{34} (1997), no.~6, 2392--2423.

\bibitem{petsc-user-ref}
Satish Balay, Shrirang Abhyankar, Mark~F. Adams, Jed Brown, Peter Brune, Kris
  Buschelman, Lisandro Dalcin, Alp Dener, Victor Eijkhout, William~D. Gropp,
  Dmitry Karpeyev, Dinesh Kaushik, Matthew~G. Knepley, Dave~A. May,
  Lois~Curfman McInnes, Richard~Tran Mills, Todd Munson, Karl Rupp, Patrick
  Sanan, Barry~F. Smith, Stefano Zampini, Hong Zhang, and Hong Zhang,
  \emph{{PETS}c users manual}, Tech. Report ANL-95/11 - Revision 3.15, Argonne
  National Laboratory, 2021.

\bibitem{petsc-web-page}
\bysame, \emph{{PETS}c {W}eb page}, \url{https://www.mcs.anl.gov/petsc}, 2021.

\bibitem{petsc-efficient}
Satish Balay, William~D. Gropp, Lois~Curfman McInnes, and Barry~F. Smith,
  \emph{Efficient management of parallelism in object oriented numerical
  software libraries}, Modern Software Tools in Scientific Computing (E.~Arge,
  A.~M. Bruaset, and H.~P. Langtangen, eds.), Birkh{\"{a}}user Press, 1997,
  pp.~163--202.

\bibitem{BaylissGoldsteinTurkel1985}
A.~Bayliss, C.I. Goldstein, and E.~Turkel, \emph{On accuracy conditions for the
  numerical computation of waves}, Journal of Computational Physics \textbf{59}
  (1985), no.~3, 396--404.

\bibitem{BeamsGillmanHewett2020}
Natalie~N. Beams, Adrianna Gillman, and Russell~J. Hewett, \emph{A parallel
  shared-memory implementation of a high-order accurate solution technique for
  variable coefficient {H}elmholtz problems}, Computers \& Mathematics with
  Applications \textbf{79} (2020), no.~4, 996--1011.

\bibitem{BeckCanzaniMarzuola2021}
Thomas Beck, Yaiza Canzani, and Jeremy~L. Marzuola, \emph{Quantitative bounds
  on impedance-to-impedance operators with applications to fast direct solvers
  for {PDE}s}, 2021.

\bibitem{Boyd2001}
J.~Boyd, \emph{Chebyshev and fourier spectral methods}, 2nd ed., Dover,
  Mineola, New York, 2001.

\bibitem{Brandt1977}
Achi Brandt, \emph{Multi-level adaptive solutions to boundary-value problems},
  Mathematics of Computation \textbf{31} (1977), no.~138, 333--390.

\bibitem{Brandt1994}
\bysame, \emph{Rigorous quantitative analysis of multigrid. i. constant
  coefficients two-level cycle with {$L^2$}-norm}, SIAM J. Numer. Anal.
  \textbf{6} (1994), no.~31, 1695--1730.

\bibitem{ChandlerWilde1997}
S.~N. Chandler-Wilde, \emph{The impedance boundary value problem for the
  helmholtz equation in a half-plane}, Mathematical Methods in the Applied
  Sciences \textbf{20} (1997), no.~10, 813--840.

\bibitem{Chaumont2015}
Th{\'e}ophile Chaumont~Frelet, \emph{{Finite element approximation of
  {H}elmholtz problems with application to seismic wave propagation}}, Theses,
  {INSA de Rouen}, Dec 2015.

\bibitem{ErnstGander2012}
O.~G. Ernst and M.~J. Gander, \emph{Why it is difficult to solve helmholtz
  problems with classical iterative methods}, pp.~325--363, Springer Berlin
  Heidelberg, Berlin, Heidelberg, 2012.

\bibitem{GanderZhang2019}
Martin~J. Gander and Hui Zhang, \emph{A class of iterative solvers for the
  helmholtz equation: Factorizations, sweeping preconditioners, source
  transfer, single layer potentials, polarized traces, and optimized schwarz
  methods}, SIAM Review \textbf{61} (2019), no.~1, 3--76.

\bibitem{GeldermansGillman2019}
P.~Geldermans and A.~Gillman, \emph{An adaptive high order direct solution
  technique for elliptic boundary value problems}, SIAM Journal on Scientific
  Computing \textbf{41} (2019), no.~1, A292--A315.

\bibitem{George1973}
Alan George, \emph{Nested dissection of a regular finite element mesh}, SIAM
  Journal on Numerical Analysis \textbf{10} (1973), no.~2, 345--363.

\bibitem{GillmanBarnettMartinsson2015}
A.~Gillman, A.~H. Barnett, and P-G. Martinsson, \emph{A spectrally accurate
  direct solution technique for frequency-domain scattering problems with
  variable media}, {BIT} Numerical Mathematics \textbf{55} (2015), no.~1,
  141--170.

\bibitem{GillmanMartinsson2014}
A.~Gillman and P.~G. Martinsson, \emph{A direct solver with \$o(n)\$ complexity
  for variable coefficient elliptic {PDE}s discretized via a high-order
  composite spectral collocation method}, SIAM Journal on Scientific Computing
  \textbf{36} (2014), no.~4, A2023--A2046.

\bibitem{GrahamSauter2020}
I.~G. Graham and S.~A. Sauter, \emph{Stability and finite element error
  analysis for the helmholtz equation with variable coefficients}, Mathematics
  of Computation \textbf{89} (2020), 105--138.

\bibitem{CalandraGrattonPinelVasseur2012}
X.~Pinel H.~Calandra, S.~Gratton and X.~Vasseur, \emph{An improved two-grid
  preconditioner for the solution of three-dimensional helmholtz problems in
  heterogeneous media}, Numer. Linear Algebra Appl. \textbf{20} (2012),
  663--688.

\bibitem{HaoMartinsson2016}
Sijia Hao and Per-Gunnar Martinsson, \emph{A direct solver for elliptic {PDE}s
  in three dimensions based on hierarchical merging of {P}oincaré–{S}teklov
  operators}, Journal of Computational and Applied Mathematics \textbf{308}
  (2016), 419--434.

\bibitem{Hemker2003}
P.~Hemker, W.~Hoffmann, and M.~van Raalte, \emph{Two-level fourier analysis of
  a multigrid approach for discontinuous {G}alerkin discretization}, SIAM
  Journal on Scientific Computing \textbf{3} (2003), no.~25, 1018--1041.

\bibitem{Hemker2004}
P.~W. Hemker, W.~Hoffmann, and M.~H. van Raalte, \emph{Fourier two-level
  analysis for discontinuous {G}alerkin discretization with linear elements},
  Numerical Linear Algebra with Applications \textbf{5 -- 6} (2004), no.~11,
  473--491.

\bibitem{Lighthill1978}
Lighthill~M. J., \emph{Waves in fluids}, Cambridge University Press, Cambridge,
  England, 1978.

\bibitem{KimLee1996}
S.~Kim and M.~Lee, \emph{Artificial damping techniques for scalar waves in the
  frequency domain}, Comput. Math. Appl. \textbf{31} (1996), 1--12.

\bibitem{KinslerFreyCoppensSanders2000}
L.~E. Kinsler, A.~R. Frey, A.~B. Coppens, and J.~V. Sanders.,
  \emph{Fundamentals of acoustics}, Wiley, 2000.

\bibitem{Knyazev2001}
Andrew~V. Knyazev, \emph{Toward the optimal preconditioned eigensolver: Locally
  optimal block preconditioned conjugate gradient method}, SIAM Journal on
  Scientific Computing \textbf{23} (2001), no.~2, 517--541.

\bibitem{KronbichlerKormann2019}
Martin Kronbichler and Katharina Kormann, \emph{Fast matrix-free evaluation of
  discontinuous galerkin finite element operators}, ACM Trans. Math. Softw.
  \textbf{45} (2019), no.~3.

\bibitem{KyurkchanSmirnova2016}
Alexander~G. Kyurkchan and Nadezhda~I. Smirnova, \emph{Mathematical modeling in
  diffraction theory}, Elsevier, Amsterdam, 2016.

\bibitem{LuceroGander2020}
José~Pablo Lucero~Lorca and Martin~Jakob Gander, \emph{Optimization of
  two-level methods for dg discretizations of reaction-diffusion equations},
  2020.

\bibitem{Martinsson2009}
P.G. Martinsson, \emph{A fast direct solver for a class of elliptic partial
  differential equations}, Journal of Computational Physics \textbf{38} (2009),
  316--330.

\bibitem{Martinsson2013}
\bysame, \emph{A direct solver for variable coefficient elliptic {PDE}s
  discretized via a composite spectral collocation method}, Journal of
  Computational Physics \textbf{242} (2013), 460--479.

\bibitem{PfeifferKidderScheelTeukolsky2003}
H.P. Pfeiffer, L.E. Kidder, M.A. Scheel, and S.A. Teukolsky, \emph{A
  multidomain spectral method for solving elliptic equations}, Computer physics
  communications \textbf{152} (2003), no.~3, 253--273.

\bibitem{RiyantiKononovErlanggaVuikOosterleePlessix2007}
C.~D. Riyanti, A.~Kononov, Y.~A. Erlangga, C.~Vuik, C.~W. Oosterlee, R.-E.
  Plessix, and W.~A. Mulder, \emph{A parallel multigrid-based preconditioner
  for the 3d heterogeneous high-frequency helmholtz equation}, J. Comput. Phys.
  \textbf{224} (2007), 431--448.

\bibitem{Roth1934}
W.~E. Roth, \emph{{On direct product matrices}}, Bulletin of the American
  Mathematical Society \textbf{40} (1934), no.~6, 461 -- 468.

\bibitem{CoolsRepsVanroose2014}
B.~Reps S.~Cools and W.~Vanroose, \emph{A new level-dependent coarse grid
  correction scheme for indefinite helmholtz problems}, Numer. Linear Algebra
  Appl. \textbf{21} (2014), 513--533.

\bibitem{SaadSchultz1986}
Youcef Saad and Martin~H. Schultz, \emph{{GMRES}: A generalized minimal
  residual algorithm for solving nonsymmetric linear systems}, SIAM Journal on
  Scientific and Statistical Computing \textbf{7} (1986), no.~3, 856--869.

\bibitem{Sanchez-Palencia1980}
Enrique Sanchez-Palencia, \emph{Non-homogeneous media and vibration theory},
  Lecture Notes in Physics, Springer-Verlag, Berlin, Heidelberg, 1980.

\bibitem{SmithBjorstadGropp1996}
Barry~Francis Smith, Petter~E. Bj{\o}rstad, and William~D. Gropp, \emph{Domain
  decomposition : parallel multilevel methods for elliptic partial differential
  equations}, Cambridge University Press, Cambridge, 1996.

\bibitem{ToselliWidlund2005}
Andrea Toselli and Olof~B. Widlund, \emph{Domain decomposition methods :
  algorithms and theory}, Springer series in computational mathematics,
  Springer, Berlin, 2005.

\bibitem{Trefethen2000}
Lloyd~N. Trefethen, \emph{Spectral methods in matlab}, Society for Industrial
  and Applied Mathematics, 2000.

\bibitem{ErlanggaVuikOosterlee2004}
C.~Vuik Y.~A.~Erlangga and C.~W. Oosterlee, \emph{On a class of preconditioners
  for the helmholtz equation}, Appl. Numer. Math. \textbf{50} (2004), 409--425.

\end{thebibliography}

\end{document}